\documentclass[a4paper,manuscript]{amsproc}
\usepackage{amssymb}
\usepackage[hyphens]{url} 
\usepackage{tikz}
\usepackage{tikz-cd}
\usepackage{extarrows}
\usepackage{tkz-graph}
\usepackage{url}

\theoremstyle{plain}
 \newtheorem{theorem}{Theorem}[section]
 \newtheorem{proposition}{Proposition}[section]
 \newtheorem{conjecture}{Conjecture}[section]
\theoremstyle{definition}
 \newtheorem{definition}{Definition}[section]
\theoremstyle{remark}
 \newtheorem{remark}{Remark}[section]
 \numberwithin{equation}{section}

\newtheorem{example}{Example}[section]

\renewcommand{\leq}{\leqslant}
\renewcommand{\geq}{\geqslant}

\newcommand{\Ind}{\mathrm{Ind}}
\newcommand{\Syl}{\mathrm{Syl}}

\newcommand{\Mod}{\text{Mod}}
\newcommand{\Hom}{\mathrm{Hom}}
\newcommand{\ZZ}{\mathbb{Z}}
\newcommand{\GG}{\mathbb{G}}
\newcommand{\QQ}{\mathbb{Q}}
\newcommand{\NN}{\mathbb{N}}
\newcommand{\Uch}{\mathrm{Uch}}
\newcommand{\Arg}{\text{Arg}}

\title[Brou\'e's Conjecture for $\Omega^{+}_8(2)$]{An algorithmic approach to perverse derived equivalences  Brou\'e's Conjecture for $\Omega^{+}_8(2)$}

\author[Stefano Sannella]{\bfseries Stefano Sannella} 

\address{ 
Department of Mathematics \\ 
University of Kaiserslautern   \\ 
Germany}

\begin{document}

{\begin{flushleft}\baselineskip9pt\scriptsize

\end{flushleft}}
\vspace{18mm} \setcounter{page}{1} \thispagestyle{empty}

\vspace{3mm}

\begin{abstract}
Following Craven and Rouquier's method \cite{Dav} to tackle Brou\'e's abelian defect group conjecture, we present two algorithms implementing that procedure in the case of principal blocks of defect $D \cong C_{\ell} \times C_{\ell}$ for a prime $\ell$; the first describes a stable equivalence between $B_0(G)$ and $B_0(N_G(D))$, and the second tries to lift such a stable equivalence to a perverse derived equivalence. 
As an application, we show that Brou\'e's conjecture holds for the principal $5$-block of the simple group $\Omega^{+}_8(2)$.
\end{abstract}

\maketitle


\section{Introduction} 

Let $G$ be a finite group and $\ell$ be a prime number. In representation theory, it is an open problem to determine how and why some aspects of the representation theory of $G$ are somehow controlled by the family of subgroups of the form $N_G(Q)$, where $Q$ is a non-trivial $\ell$-subgroup of $G$; such subgroups are called $\ell$-\emph{local}. In \cite{Br}, Michel Brou\'e conjectured the following structural connection, which would imply other notable local/global conjectures:

\begin{conjecture}\textbf{\emph{(Brou\'e's abelian defect group conjecture - 1990, \cite{Br})}}\label{fghj}
Let $G$ be a finite group and $\ell$ be a prime number. Let $B$ be an $\ell$-block of $G$ with abelian defect group $D$ and $b$ be the $\ell$-block of $N_G(D)$ corresponding to $B$ under the Brauer correspondence; then $B$ and $b$ are derived equivalent.
\end{conjecture}
A notable refinement of Brou\'e's conjecture was proposed by Rickard \cite{Rick}: 
\begin{conjecture}\textbf{\emph{(Brou\'e-Rickard)}}\label{splee}
With the same notation and under the same conditions stated in Brou\'e's abelian defect conjecture, there is a splendid equivalence between $B$ and $b$.
\end{conjecture}

Brou\'e's conjecture has been checked in many different cases, but we are still far from a complete solution. Some cases that have been proved include $\ell$-blocks of individual groups, for example some sporadic groups, as well as families of groups, like alternating and symmetric groups. In this article we show that:

\begin{theorem}\label{mine}
Rickard's refinement of Brou\'e's abelian defect group conjecture holds for the principal $5$-block of $\Omega^{+}_{8}(2)$. Moreover, the derived equivalence can be chosen to be a perverse equivalence. 
\end{theorem}

\vspace{1mm}

Our computational strategy is based on two algorithms: \texttt{FinalStabEq} and \texttt{PerverseEq}. Both algorithms are available online, together with the several sub-algorithms having a role inside these two main ones. In order to provide an introductory view of how this algorithmic approach works, and before going into the details, we can generically state that:

\begin{itemize}
\item \texttt{FinalStabEq} is determining the image of each simple $B_0(G)$-module $S$ under a stable equivalence $L : \underline{\text{mod}}(B_0(G))  \xrightarrow{\sim} \underline{\text{mod}}(B_0(H))$ which was originally constructed by Rouquier in \cite[\S 5.5]{Rouq}.
\item \texttt{PerverseEq} is returning, for each simple $B_0(H)$-module $T$, a complex $X_T$ of $B_0(H)$-modules; if those complexes fulfil the conditions of Proposition \ref{forestanera}, then they consist of the image of the simple $B_0(G)$-modules under a derived perverse equivalence.
\end{itemize}
As stated in Proposition \ref{forestanera}, we are able to deduce that there is a perverse derived equivalence between $B_0(G)$ and $B_0(H)$ if there is a bijection $S \xrightarrow{ \ \ 1 : 1 \ \ } T$ between simple $B_0(G)$ and $B_0(H)$-modules such that $L(S)$ and $X_T$ are isomorphic in the stable category. 
This procedure is based on a result in the theory of perverse equivalences in the setting of Brou\'e's abelian defect group conjecture developed by Rouquier and Chuang; the result in more generality can be found in \cite[\S 3]{Dav}. 

\subsection{Notation}

Every group is intended to be finite; $\ell$ will denote a prime number.
As usual in representation theory, we will denote by $(K,\mathcal{O},k)$ an $\ell$-modular system, i.e. $\mathcal{O}$ is a complete discrete valuation ring such that $\mathcal{O}/J(\mathcal{O}) \simeq k$ is a field of characteristic $\ell$ and $K$ is the field of fractions of $\mathcal{O}$. The ring $\mathcal{O}$ will always be assumed to be large enough for the group that we consider; we will also assume that $\mathcal{O}$ is an extension of $\ZZ_{\ell}$ so that $\QQ_{\ell} \subseteq K$. 

An $\ell$-local subgroup of $G$ will often be denoted by $H$, and in the last two sections we will always have $H=N_G(D)$, where $D \in \Syl_{\ell}(G)$. For a group $G$, we denote by $B_0(G)$ the principal $\ell$-block of $G$ over $k$; for a general $\ell$-block $B$ of $G$ over $k$, the Brauer correspondent block of $H$ will be denoted by $b$, whereas for principal blocks we will simply use $B_0(H)$. 

As for modules, given an $\ell$-block $B$, we will refer to a complete set of representatives of isomorphism classes of simple $B$-modules by $\mathcal{S}_B$. Every module is intended to be a left module. For a $kG$-module $M$, the restriction of $M$ down to $H$ will be denoted by $M_H$; for a $kH$-module $V$, the induction up to $G$ is $\Ind_H^G V$. For a module $M$, the projective cover of $M$ is denoted by $\mathcal{P}(M)$. The set of composition factors of $M$ is denoted by $\text{cpf}(M)$.

Finally, for a block $B$, $\Mod(B)$ is the category consisting of all $B$-modules and $\text{mod}(B)$ the subcategory of all finitely generated $B$-modules; we will denote by $\mathcal{D}(B)$ the derived category of $\text{mod}(B)$ of \emph{bounded} complexes. The stable category of $\text{mod}(B)$ will be denoted by $\underline{\text{mod}}(B)$. We will say that $B$ and $b$ are derived equivalent if $\mathcal{D}(B)$ and $\mathcal{D}(b)$ are equivalent categories, and stably equivalent if $\underline{\text{mod}}(B)$ and $\underline{\text{mod}}(b)$ are equivalent categories. Stable equivalences and derived equivalences are always regarded as equivalences of triangulated categories. 

\section{Strategy and algorithm}\label{saa}

\subsection{From stable to derived equivalences}

In the last thirty years, Brou\'e's Conjecture has been tackled by using a wide range of methods. Although a general and uniform solution seems far from being achieved, these efforts have provided results for several restricted cases. They include the following important one:
\begin{theorem}\emph{\cite{Link1, Rick1, Rouq1}}\label{fgeom}
Rickard's refinement of Brou\'e's abelian defect group conjecture holds when the defect group $D$ is cyclic.
\end{theorem}

In particular, the conjecture holds for $D \simeq C_{\ell}$. It is natural to ask if we can deduce something for the next step $D \simeq C_{\ell} \times C_{\ell}$, and especially if we can use the known case $D \simeq C_{\ell}$ together with an inductive strategy. Although we do not have a solution for the rank $2$ case in general, this approach was successfully performed in some situations (including the result of this article). Let us say something more about it. 

\vspace{2mm}

Let $Q$ be an $\ell$-subgroup of $G$. We recall that the Brauer map 
\begin{center}
$\text{Br}_Q : \text{Mod}(kG) \to \text{Mod}(k N_G(Q))$ 
\end{center}
is defined on the objects as $\text{Br}_Q : M \mapsto M^Q / (\sum_{R < Q} \text{Tr}_R^Q M^R)$, where $M^R$ is the set of points fixed by $R$ and the trace map $\text{Tr}_R^Q : M^R \to M^Q$ is defined as $\text{Tr}_R^Q(m):= \sum_{g \in Q/R} gm$, for $m \in M^R$. 

\vspace{1mm}

An interesting case of application for the Brauer map arises when $\text{Br}_Q$ is restricted to a specific class of modules, namely the $\ell$-\emph{permutation modules}. We recall that a $kG$-module $M$ is an $\ell$-permutation module if for every $\ell$-subgroup $Q$, there is a basis of $M$ which is invariant under $Q$. Equivalently, we can define $\ell$-permutation modules as direct sums of trivial source modules.

The specific case of the Brauer map acting on $\ell$-permutation modules is explicitly treated in \cite[\S 4.1]{Rouq}; in the following, the Brauer map will always be assumed to act on such class of modules, and this allows us to extend the action $\text{Br}_Q$ to complexes of ($\ell$-permutation) bimodules. Under such condition, the Brauer map becomes a functor. Following \cite{Rouq}, we say that a complex of $(B_0(H),B_0(G))$-bimodules $C$ is splendid if, regarded as an $k[H \times G]$-module, each indecomposable summand of each term of $C$ is an $\ell$-permutation module, and its vertices are contained in $\Delta D$. The property for a complex to be splendid is crucial in Proposition \ref{induc}, and is the core of Rickard's refinement of Brou\'e's conjecture, as stated in Conjecture \ref{splee}. We refer to \cite{Rick} for the detailed definition of splendid equivalence.  

\vspace{1mm}

The following result connects the existence of a derived equivalence at the local level and a stable equivalence at the global level:
\begin{proposition}\textbf{\emph{(Rouquier, \cite[Th. 5.6]{Rouq}}}\label{induc}
Let $C$ be a splendid complex of $(B_0(H),B_0(G))$-bimodules. The following two assertions are equivalent:
\begin{itemize}
\item $C \otimes _{kG} -$ induces a splendid stable equivalence between $B_0(G)$ and $B_0(H)$;
\item for every non-trivial subgroup $Q$ of $D$, the complex $\emph{Br}_{\Delta Q}(C)$ induces a splendid derived equivalence between $B_0(C_G(Q))$ and $B_0(C_H(Q))$,
\end{itemize}
where $\emph{Br}$ is the Brauer map extended to complexes of modules.
\end{proposition}
\begin{remark}
In this result and in the rest of the article, we usually consider the stable category as a quotient of the derived category: in \cite{Rick1}, Rickard shows the existence of an equivalence 
$\mathcal{D}(B) / \mathcal{K}(\text{proj-} B) \to \underline{\text{mod}}(B)$, where $\mathcal{K}(\text{proj-} B)$ is the bounded homotopy category of $\text{proj-} B$; therefore suitable complexes can induce equivalences at the level of the stable category. This equivalence is also used to show that a derived equivalence induces a stable equivalence. 
\end{remark}

\begin{remark}
In general, a derived equivalence between blocks over $\mathcal{O}$ implies a derived equivalence between the same blocks over $k$, but the converse is not true; however, this is true for splendid derived equivalences. This property is especially useful in our computational procedure, since calculation are run over fields rather than over $\mathcal{O}$. Therefore, in the following we will always let $\ell$-blocks to be intended over $k$. 
\end{remark}

Still in \cite[\S 5.5]{Rouq}, whenever $D \cong C_{\ell} \times C_{\ell}$, Rouquier applies this result to build a complex $C$ of $(B_0(H),B_0(G))$-bimodules inducing a stable equivalence. The strategy consists of building complexes of $kN_{H \times G}( \Delta Q)$-modules such that the restriction to $C_{H}(Q) \times C_G(Q)$, seen as a $(kC_{H}(Q),kC_G(Q))$-bimodule, induces a derived equivalence between $B_0(C_G(Q))$ and $B_0(C_H(Q))$, for each conjugacy class of non-trivial $Q < D$. The construction of such modules relies on the knowledge of the derived equivalence when the defect group is cyclic.

In particular, this strategy applies favourably for $\ell=2$, and it is used to prove Brou\'e's conjecture in general when $D \cong C_2 \times C_2$ (again, in \cite[\S 6.1.1]{Rouq}). When $\ell$ is odd (in our case $\ell=5$), such a general result to lift the stable equivalence induced by $C$ to a derived equivalence does not work, but the construction of the stable equivalence still holds and lifting this particular stable equivalence to a derived equivalence can be attempted case by case. 

In our computational setting, we mostly care about the image in $\mathcal{D}(B_0(H))$ of the stable equivalence induced by the complex $C$; in particular, we will compute the image of each simple module $S \in \mathcal{S}_{B_0(G)}$ under this stable equivalence and we will compare it with the output of the perverse equivalence algorithm corresponding to $S$ (namely, the image of $S$ under the supposed perverse equivalence). Together with the construction of $C$ in \cite[\S 5.5]{Rouq}, in \cite[\S 3.3.1]{Dav} we have an explicit construction of what the image $C \otimes_{kG} S$ is as a complex of $kH$-modules: this is the object that we need to know, and that we will physically build in our algorithm \texttt{FinalStabEq}. More precisely, this complex has length $2$; the module in degree $0$ being given by the Green correspondence (that we already have, in most cases), the algorithm will actually build the module in degree $-1$. In Section \ref{decc} we focus on how the construction of each image $C \otimes_{kG} S$ is performed.

\subsection{The algorithm \texttt{FinalStabEq}}\label{decc} In this section we recall the actual construction of the image under the splendid stable equivalence $L:=C \otimes _{kG} -$ between $B_0(G)$ and $B_0(H)$ descending from Proposition \ref{induc}; this construction can be found in \cite[\S 3.3.2]{Dav}. In other words, we give a simplified expression of the complexes of $B_0(H)$-modules:
\begin{equation}
\{ C \otimes _{kG} S \mid S \in \mathcal{S}_{B_0(G)} \},
\end{equation}
with the aim to construct them in MAGMA \cite{Mag}; the actual construction of the complex of $(B_0(H),B_0(G))$-bimodules $C$ is not computationally feasible, and in fact it is not necessary for our method, for we are interested in the set of complexes of $B_0(H)$-modules $C \otimes _{kG} S$ above only. 

\vspace{2mm}

Let us fix a subgroup $Q < D \in \Syl_{\ell}(G)$ of order $\ell$. We are still assuming that $D \cong C_{\ell} \times C_{\ell}$. In the following, we will assume that there exist $\bar{N}_G(Q)$ and $\bar{N}_H(Q)$ which denote complements of $Q$ in $N_G(Q)$ and $N_H(Q)$ respectively; those complements exist for the case that we will consider. In particular, we choose them such that $\bar{N}_H(Q) \leq \bar{N}_G(Q)$. Finally, let us consider $\bar{C}_H(Q)=C_H(Q) \cap \bar{N}_H(Q)$ and $\bar{C}_G(Q)=C_G(Q) \cap \bar{N}_G(Q)$; therefore, both $\bar{C}_G(Q)$ and $\bar{C}_H(Q)$ have cyclic Sylow $\ell$-subgroups, and we have a derived equivalence between their principal $\ell$-blocks.

We set $N_{\Delta}:= (\bar{C}_H(Q) \times \bar{C}_G(Q)^{\text{opp}}) \Delta \bar{N}_H(Q)$; this group acts on $\bar{C}_G(Q)$ and then we can consider $k\bar{C}_G(Q)$ as a $k N_{\Delta}$-module as well as a $\bar{C}_H(Q) \times \bar{C}_G(Q)^{\text{opp}}$-module. Let $e_{\bar{C}_H(Q)}$ and $e_{\bar{C}_G(Q)}$ denote the block idempotents corresponding to the principal blocks of $\bar{C}_H(Q)$ and $\bar{C}_G(Q)$ respectively. 
As $k N_{\Delta}$-module we have that
\begin{equation}\label{propro}
e_{\bar{C}_H(Q)} k \bar{C}_G(Q) e_{\bar{C}_G(Q)} =M _Q \oplus P,
\end{equation}
where $M_Q$ is indecomposable as a $\bar{C}_H(Q) \times \bar{C}_G(Q)^{\text{opp}}$-module and induces a stable equivalence (Rouquier, \cite[\S 5.5]{Rouq} or \cite[\S 3.3.1]{Dav}), whereas $P$ is projective. We have a precise description of what a projective cover of $M_Q$ is isomorphic to. Let $V \in \mathcal{S}_{B_0(\bar{C}_G(Q))}$. We consider the map $\gamma : \mathcal{S}_{B_0(\bar{C}_G(Q))} \to \mathcal{S}_{B_0(\bar{C}_H(Q))}$, where $\gamma(V)$ is defined to be the unique simple $B_0(\bar{C}_H(Q))$-module such that 
\begin{equation*}
\underline{\Hom}(V_{\bar{C}_H(Q)}, \gamma(V)) \neq \{0\}.
\end{equation*}

A projective cover of $M_Q$ as a $\bar{C}_H(Q) \times \bar{C}_G(Q)^{\text{opp}}$-module is of the form 
\begin{equation}\label{penn}
\mathcal{P}(M_Q) \cong \bigoplus_{V \in \mathcal{S}_{B_0(\bar{C}_G(Q))}} \mathcal{P}(\gamma(V)) \otimes \mathcal{P}(V)^{*},
\end{equation}
where each summand $\mathcal{P}(\gamma(V)) \otimes \mathcal{P}(V)^{*}$ gains the natural structure of  $\bar{C}_H(Q) \times \bar{C}_G(Q)^{\text{opp}}$-module. Finally, we define the subset $\mathcal{E} \subset \mathcal{S}_{B_0(\bar{C}_G(Q))}$ of all modules whose corresponding edge in the Brauer tree of $B_0(\bar{C}_G(Q))$ has distance $d+1 \ (\text{mod} \ 2)$ from the exceptional vertex, where $d$ is the distance of the trivial module from the exceptional vertex. In other words, depending on $d$, we consider the set of edges whose distance from the exceptional node is even or odd. We now define $U_Q :=\bigoplus_{V \in \mathcal{E}} \mathcal{P}(\gamma(V)) \otimes \mathcal{P}(V)^{*}$; again by \cite[\S 3.3.1]{Dav}, it is possible to extend the action of $\bar{C}_H(Q) \times \bar{C}_G(Q)^{\text{opp}}$ up to $N_{\Delta}$; with an abuse of notation, we will see $U_Q$ as a $k N_{\Delta}$-module from now on. We define $T_Q:=U_Q \oplus P$, where $P$ is the projective $k N_{\Delta}$-module appearing in the decomposition (\ref{propro}). 
\begin{remark}
The module $M_Q$ is what we are building in a sub-algorithm \texttt{StableEqSetup}, together with all the necessary groups and subgroups involved, such as $\bar{C}_H(Q), \bar{C}_G(Q), N_{\Delta}$; the module $T_Q$ is built manually case by case, since the construction depends on the Brauer tree of $\bar{C}_G(Q)$; $T_Q$ will be given as an input to the algorithm \texttt{FinalStabEq}.
\end{remark}
It remains to explain how to use these objects to get the complex of $kH$-modules $C \otimes_{kG} S$. The tensor product $T_Q \otimes_{k \bar{C}_G(Q)} S_{N_G(Q)}$ gains the structure of $N_{\Delta} \times N_G(Q)$-module, and in particular we will regard it as a $N_H(Q)$-module: the copy of $N_H(Q)$ inside $N_{\Delta} \times N_G(Q)$ that we consider is defined by the bijection $h \to ((\bar{h}, \bar{h}^{-1}),h)$, where $\bar{h}$ is the image of $h$ in $\bar{N}_H(Q)$; in our algorithm, $\bar{N}_H(Q)$ is constructed as a subgroup of $N_H(Q)$ such that $N_H(Q) =  Q \rtimes  \bar{N}_H(Q)$ rather than as a quotient, and therefore $\bar{h}$ will have to be defined as the element such that $h \cdot \bar{h}^{-1} \in Q$. Regarding $T_Q \otimes_{k \bar{C}_G(Q)} S_{N_G(Q)}$ as a $k N_H(Q)$-module, we can finally give the expression for $C \otimes_{kG} S$:

\begin{equation}\label{finalC}
C \otimes_{kG} S \cong  (0 \to e_{H} \bigoplus_{Q < D} \Ind_{N_H(Q)}^H (T_Q \otimes_{k \bar{C}_G(Q)} S_{N_G(Q)}) \to e_{H} S_H \to 0),
\end{equation}
where $Q$ runs over all the $H$-conjugacy classes of subgroups of $D$ of order $\ell$ and $e_H$ denotes the principal block idempotent of $kH$.

\begin{remark}\label{thesamm}
In the construction of $C \otimes_{kG} S$ above, as an object in the stable category the module in degree $0$ consists of the Green correspondent of $S$ together with the relatively projective summands occurring in the correspondence. Again by \cite[\S 3.3.1]{Dav}, it is possible to construct a complex $C'$ which is homotopy equivalent to $C$ and such that $C' \otimes_{kG} S$ has the Green correspondent only as a term of degree $0$. In particular, the term of degree $-1$ in $C'$ is constructed as for the one of $C$, but $T_Q$ is replaced by $U_Q$. 
\end{remark}

\subsection{Perverse equivalences}

In recent years, the theory of perverse equivalences has been successfully applied to gain some progress in the study of Brou\'e's conjecture, especially in the case of finite groups of Lie type in non-defining characteristic. In this section we present an algorithm which attempts to produce perverse derived equivalences for principal blocks whose defect group is elementary abelian of rank $2$; these derived equivalences are actually splendid derived equivalences. The algorithmic approach to perverse equivalences has already been used in \cite[\S 5, \S 6]{Dav} to produce perverse equivalences for some groups of Lie type as well as some sporadic groups. 
\begin{definition}{\textbf{(Perverse equivalence)}}\label{defperve}
Let $A_1,A_2$ be $k$-algebras and $F: \mathcal{D}(A_1) \to \mathcal{D}(A_2)$ be a derived equivalence. Let us denote by $S_1, \dots, S_n$ and $T_1, \dots, T_n$ as representatives of the isomorphism classes of the simple $A_1$-modules and $A_2$-modules respectively and let $\pi : \{1, \dots, n \} \to \ZZ_{\geq 0}$ be a function. We say that $F$ is a \emph{perverse equivalence} with perversity function $\pi$ with respect to the bijection between $\mathcal{S}_{A_1}$ and $\mathcal{S}_{A_2}$ sending $S_i$ to $T_i$ if for every $i  \in  \{1, \dots, n \}$ the modules occurring as composition factors of $H^{-j}(F(S_i))$ are $T_{\alpha}$ such that $\pi(\alpha) < j \leq \pi(i)$, and a single copy of $T_i$ if $j=\pi(i)$.
\end{definition}

The notion of perverse equivalence was originally developed by Chuang and Rouquier; in their work \cite{Rchu}, a more general definition (with respect to our Definition \ref{defperve}) of perverse equivalence is provided, as well as a particular mention about the specific case of non-decreasing perversity. 

We notice that Definition \ref{defperve} carries a bijection between the set of simple $A_1$-modules and $A_2$-modules, and this is given by indexing those sets from $1$ to $n$. With an abuse of notation, we can often think of $\pi$ as a function defined on the set $\{T_1, \dots, T_n \}$ rather than $\{1, \dots, n \}$, and therefore write $\pi(T_i)$ instead of $\pi(i)$. This will make the notation easier in some settings. 

\vspace{2mm}

\subsubsection{The algorithm \texttt{\emph{PerverseEq}}}\label{atee}

This algorithm is meant to construct complexes $X_1, \dots, X_n$ which are the images of the simple $B_0(G)$-modules $S_1, \dots, S_n$ under a potential perverse equivalence. 

Let $\pi : \mathcal{S}_{B_0(H)}\to \ZZ_{\geq 0}$ be a given map. Details about this map are provided in Section \ref{nmoved}, whereas here we want to show how $\pi$ is involved in our algorithmic construction. For any $r \in \ZZ_{\geq 0}$, we define:
\begin{equation}\label{algopi}
J_r:=\{V \in \mathcal{S}_{B_0(H)} \mid \pi(V) \leq r \}.
\end{equation} 
Let $T$ be a simple $B_0(H)$-module. We now explain how to produce the complex $X_T \in \mathcal{D}(B_0(H))$ which is supposed to be the image of $T$ under a potential perverse derived equivalence. 

\vspace{2mm}

Whenever $\pi(T)=0$, the algorithm will return the complex $X_T : 0 \to T \to 0$. Let us assume now that $n:=\pi(T) > 0$. Then we will produce a complex of length $n+1$ that we will denote by 
\begin{equation}\label{complexper}
X_T : 0 \to P_{n} \xlongrightarrow{\varphi_{n}} P_{n-1} \xlongrightarrow{\varphi_{n-1}} \dots \xlongrightarrow{\varphi_1} P_0 \to 0,
\end{equation}
where $P_0$ is in degree zero, and then $P_n$ in degree $-n$.
Before defining each module of the complex, we introduce the notation:
\begin{definition}
Let $A$ be an algebra and let $\mathfrak{T}$ be a set of simple $A$-modules and $M$ an $A$-module. The  $\mathfrak{T}$-\emph{radical} of $M$ is defined as the largest submodule $\mathfrak{T}$-\text{rad}$(M) \subseteq M$ with composition factors in $\mathfrak{T}$.
\end{definition}
Now we can finally build the complex \ref{complexper}. The first module $P_{n}$, of degree $-n$, will be the injective hull of $T$, so $P_{n}:=I(T)$. In particular, this is isomorphic to the projective cover of $T$, since a finitely generated $kH$-module is injective if and only if it is projective (for example, see \cite[Th. 4.11.3]{Link}). To define the map $\varphi_{n}$, we will start by giving the kernel of it. We define $\ker \varphi_{n}:=M_{n}$, where $M_{n}$ is the submodule of $P_n$ such that 
\begin{equation*}
M_n/T = J_{n-1}\text{-rad}(P_n/T).
\end{equation*}
The following term is defined by $P_{n-1} := I(P_n/M_n)$, with natural map $\varphi_n$ being the composition of the projection to the quotient and the inclusion in the injective hull. 
This is just the first step of the inductive process: in general, we set $L_i := \text{Im}(\varphi_i)$ and we get modules $M_i$ such that
\begin{equation}\label{bayes}
M_i/L_{i+1}= J_{i-1}\text{-rad}(P_{i}/L_{i+1}), \ 2 \leq i \leq n-1.
\end{equation}

This allows us to define each module of the complex inductively: we define $P_{i-1}$ and the map $\varphi_i$ as
\begin{equation*}
\begin{split}
P_{i-1} := & I(P_i/M_i), \ i=3, \dots, n-1,  \\ 
 \varphi_i: P_i &\twoheadrightarrow   P_i/M_i \hookrightarrow P_{i-1}, 
 \\
\end{split}
\end{equation*}
where the surjective map and the injective map are the natural projection and the natural inclusion in the injective hull, respectively. So far, we have defined the construction of our complex $X_T$ up to the degree $-2$, and this is what the algorithm \texttt{PerverseEq} is actually computing; it remains to define the last two terms $P_1$ and $P_0$, namely the final part of $X_T$:
\begin{equation*}
\dots \xlongrightarrow{\varphi_{2}}  P_1 \xlongrightarrow{\varphi_{1}} P_0 \to 0.
\end{equation*}
This construction can generally be performed manually, otherwise we might use the short algorithm called \texttt{FindP1}. 
The following definition is necessary to describe the condition that $P_1$ has to satisfy.

\begin{definition}\label{stk}
Let $M$ be a $kG$-module. We say that $M$ is \emph{stacked relatively projective} with respect to a single subgroup $Q$ of $G$ if $M$ admits a filtration $\{ 0 \} = M_0 \subseteq \dots \subseteq M_m=M$ for some $m \in \NN$ such that $M_j/M_{j-1}$ is relatively $Q$-projective for each $j=1,\dots, m$.
\end{definition}
In particular, a relatively $Q$-projective module, and therefore a projective module as well, is a trivial example of stacked relatively projective module. Finally, let us consider a bijection between $\mathcal{S}_{B_0(H)}$ and $\mathcal{S}_{B_0(G)}$, and let $S$ be the $B_0(G)$-module corresponding to $T$ under this bijection. We have to specify that this bijection is usually defined via trial and error, under the criterion that it must fulfil the requirements that we are going to mention. In particular, this will be the bijection which is carried in the definition of the perverse equivalence.  
Given a bijection, in order to declare the algorithm successful, we request that $P_1$ and $P_0$ fulfil the following two conditions:
\begin{enumerate}
\item $P_0$ must be a copy of $C_S$, the Green correspondent of $S$;
\item each indecomposable summand of $P_1$ must be stacked relatively projective with respect to some proper subgroup $Q < D$, which could be a cyclic group of order $\ell$, or the trivial subgroup in case of a projective summand.
\end{enumerate}
Moreover, the conditions on the cohomology which are imposed by the perverse equivalence, and which are implemented by the relation (\ref{bayes}), must also hold. In order to fulfil this cohomology condition, we build $M_1$ in the same was as each previous kernel $M_2, M_3, \dots, M_n$, so as a submodule of $I(P_2/M_2)$; however, rather than defining $P_1$ as $I(P_2/M_2)$, we try to build it as an extension of $C_S$ by $M_1$ whose summands satisfy the condition (2) above. 
The last module $P_0$ is defined as 
\begin{equation*}
P_0 := P_1/M_1 \cong C_S,
\end{equation*}
where the isomorphism to $C_S$ holds by construction of $P_1$; this would fulfil the requested condition (1) on $P_0$. The map $\varphi_1 : P_1 \to P_0$ is the natural projection to the quotient. 

\begin{remark}\label{P1}
The crucial stage of this algorithm is about the construction of $P_1$. The construction of all the previous terms $P_2, \dots, P_n$ is determined by an iterative process, whereas the construction of $P_1$ is subject to the existence of a non-trivial peculiar extension of $C_S$ by $M_1$. The existence of such an extension is basically determining whether the algorithm is working with the given datum of $\pi$ and with the chosen bijection between $\mathcal{S}_{B_0(H)}$ and $\mathcal{S}_{B_0(G)}$.
\end{remark}

\begin{remark}
An algorithm to test whether the module $P_1$ admits a filtration by a given list of modules is provided. In the case that we consider, for each conjugacy class of subgroups $Q \simeq C_{\ell}$ of $H$, it is enough to consider the list of indecomposable modules with vertex $Q$ and trivial source, namely the indecomposable summands of $\Ind^{H}_Q k$.
\end{remark}

\begin{remark}
The complexes $X_T$ that we have constructed above are actually always the image of the simple modules of a suitable symmetric algebra $A$ under a derived equivalence $\mathcal{D}(A) \simeq \mathcal{D}(B_0(H))$. This strategy would actually move the difficult point to proving that $A$ and $B_0(G)$ are Morita equivalent; more details can be found in \cite[\S 3]{Dav}.
\end{remark}

We have defined the algorithmic construction of the set of complexes $\{ X_T \mid  T \in \mathcal{S}_{B_0(H)} \}$ which are the image of a perverse equivalence (provided that they fulfil the condition of Proposition \ref{forestanera}). We conclude this section by mentioning a property of the cohomology $H(X_i)$ of each complex; this property is explained in \cite{Dav2} via an explicit example. Let us fix a simple $B_0(H)$-module $T_i$, and let $X_i:=X_{T_i}$ be the complex generated by our algorithm. We consider the following virtual module:
\begin{equation*}
\bigoplus_{j=0}^{\pi(T_i)} \left( \bigoplus_{T \in \text{cpf}(H^{-j}(X_i))} (-1)^{j-\pi(T)} T \right).
\end{equation*}
Following \cite[\S 6.1]{Dav2}, this is called the \emph{alternating sum of cohomology}; we explain how this virtual module can be used to reconstruct the unitriangular form of the decomposition matrix that we have been using to define the bijection between a subset of irreducible characters of $G$ lying in $B_0(G)$ (if $G$ is of Lie type, this is the set of unipotent characters) and $\mathcal{S}_{B_0(G)}$. For a simple module $T_m \in \mathcal{S}_{B_0(H)}$, we denote by $a_m$ its multiplicity into the alternating sum of cohomology; in particular, we notice that due to the construction of $X_i$ (which is coming from the definition of perverse equivalence), each module $T_m$ appearing in the alternating sum, i.e. with $a_m > 0$, is such that $\pi(T_m) \leq \pi(T_i)$, and the equality occurs when $m=i$. In the following we denote by $r_j$ the vector consisting of the $j$-th row of the fixed unitriangular decomposition matrix, and by $v_j$ the vector of length $|\mathcal{S}_{B_0(H)}|$ consisting of $0$ in each entry, except for the $j$-th entry, which is $1$. As explained in \cite[Example 6.2]{Dav2}, the numbers $a_m$ fulfil the following conditions:
\begin{equation}\label{annina}
v_i = \sum_{\substack{m : \\  \pi(T_m) \leq \pi(T_i)}} a_m \cdot r_m=a_i \cdot r_i+ \sum_{\substack{m : \\  \pi(T_m) < \pi(T_i)}} a_m \cdot r_m.
\end{equation}
In particular, the rows of the decomposition matrix that we are considering have been ordered according to the $\pi$-value of each irreducible character, therefore each row $r_m$ such that $\pi(T_m) < \pi(T_i)$ comes before $r_i$; for example, the row of the trivial character $1_G$, whose $\pi$-value is $0$, is always at the top of the matrix. The relations (\ref{annina}) show that we can reconstruct the unitriangular decomposition matrix inductively: assume that we already have each row $r_m$ such that $\pi(T_m) < \pi(T_i)$, then the alternating sum of cohomology would provide the numbers $a_m$ for $m=1, \dots, i$ and therefore we can compute the next row $r_i$ of the decomposition matrix. 

In our example $G=\Omega^{+}_8(2)$ (see the table \ref{scle}) , we will report the data coming from the alternating sum of cohomology under the label ``total'', and using the formal expression
\begin{equation*}
\sum_{\substack{m : \\  \pi(T_m) \leq \pi(T_i)}} a_m \cdot m
\end{equation*}
instead of the vector notation with $r_m$. Typically, we will almost always find that our coefficients $a_m$ are $1$, $-1$ or $0$. 

\vspace{3mm}

\section{Rouquier's theorem}

We are now ready to state the theorem which explains how the results of the algorithm can let us deduce Brou\'e's conjecture for the given $B_0(G)$.

\begin{proposition}\emph{\textbf{(\cite[\S 3]{Dav})}}\label{forestanera}
Let $G$ be a finite group, $\ell$ a prime number, $D \in \Syl_{\ell}(G)$, where $D \simeq C_{\ell} \times C_{\ell}$, and $H:=N_G(D)$. Let $L : \underline{\emph{mod}}(B_0(G))  \xrightarrow{\sim} \underline{\emph{mod}}(B_0(H))$ be the stable equivalence described in Section \ref{decc}. Let us assume that: 
\begin{itemize}
\item 
there is a perversity function $\pi$ and there is a bijection between $\mathcal{S}_{B_0(G)}$ and $\mathcal{S}_{B_0(H)}$ such that for each $T \in \mathcal{S}_{B_0(H)}$, the complex $X_T$ fulfils the two conditions that make the algorithm \texttt{PerverseEq} successful (as stated in Section \ref{atee});
\item each $X_T$ is stably isomorphic to $L(S)$, where $T \in \mathcal{S}_{B_0(H)}$  and $S \in \mathcal{S}_{B_0(G)}$ correspond under the bijection introduced above.
\end{itemize}
If those two conditions hold, then there is a perverse derived equivalence between $B_0(G)$ and $B_0(H)$, and therefore Brou\'e's conjecture holds for the principal $\ell$-block of $G$. In particular, this derived equivalence between $B_0(G)$ and $B_0(H)$ induces $L$ as a stable equivalence, and if we regard $S$ as a complex concentrated in degree zero, $X_T$ is the image of $S$ under this perverse derived equivalence.
\end{proposition}

In other words, the second condition of Proposition \ref{forestanera} is what connects our two algorithms, since it states that in order to deduce Brou\'e's conjecture for $B_0(G)$, the objects that \texttt{PerverseEq} and \texttt{FinalStabEq} have independently constructed must coincide up to projective summands. 

\vspace{1mm}

In the next section, we shall give a very short review of the so-called geometric Brou\'e's conjecture, with the aim to justify more why this computational approach works well for groups of Lie type, as well as explaining where the expression given in (\ref{perversity}) of our perversity map comes from. 

\begin{remark}\label{sacrif}
The construction of the objects $C \otimes_{kG} S$ has been implemented in each case considered, using the isomorphism (\ref{finalC}). The method of constructing $C \otimes_{kG} S$ that we have just explained has proved to be successful as long as each complex $X_T$ for $T \in S_{B_0(H)}$ that is produced by \texttt{PerverseEq} has the the following property: every indecomposable summand of the module $X^{-1}_T$ in degree $-1$ has vertex $Q$ for some proper (possibly trivial) subgroup $Q < D$; the subgroup $Q$ will generally differ when considering different summands of $X^{-1}_T$.
In general, it is not true that each summand of $X^{-1}_T$ is relatively $Q$-projective: when dealing with individual groups, we will see that $X^{-1}_T$ is a sum of modules of vertex $Q$ and \emph{stacked} relatively projective modules (see Definition \ref{stk}), and whenever this last type of summand occurs, the result for $C \otimes_{kG} S$ given by the algorithm \texttt{FinalStabEq} has proved to be incorrect, and so we have not been able to conclude by applying Proposition \ref{forestanera}. This has not allowed us to complete the proof of Brou\'e's conjecture yet for $G={}^2F_4(2).2,Sp_8(2)$ with $\ell=5$, and ${}^3D_4(2)$ with $\ell=7$. Work in order to fix this result and make it produce the right images under the stable equivalence in every case is in progress.
\end{remark}

\section{Geometric Brou\'e's conjecture and perversity functions}\label{geomdm}

Although we will not provide a deep report of the current theory behind it, we can mention how the search of a perverse equivalence intersects with some underlying geometry of the group, represented by some Deligne-Lusztig varieties. In our setting, this connection (still conjectural in large part) is related to the crucial choice of the perversity function $\pi$, for which we will introduce the precise expression (\ref{perversity}) to use in our algorithm. Therefore, in this section $G=\GG(q)$ will be a group of Lie type, for some generic group of Lie type $\GG$ and some $p$-power $q$, where $p \neq \ell$ is a prime; moreover, the facts that we state here are generally valid for \emph{unipotent} $\ell$-blocks $B$ of $G$. We recall that $B$ is a unipotent block if there is a unipotent character lying in $B$. As the trivial character is unipotent, the principal block $B_0(G)$ is a unipotent block.
The set $\Uch(B_0(G))$ will denote the unipotent characters of $G$ lying in the principal block $B_0(G)$. 

\vspace{1mm}

The geometric form of Brou\'e's Conjecture predicts that the derived equivalence between a unipotent block $B$ and its Brauer correspondent $b$ is induced by the cohomology complex of a specific variety $Y_{\kappa/d}$, for some coprime integers $\kappa, d$, carrying an action of $C_G(D)$ on the left and an action of $G$ on the right. Such a variety is called \emph{Deligne-Lusztig variety}. We will not go into the details of the theory behind the Deligne-Lusztig varieties, but we can mention that the existence of a suitable (for the purpose of Brou\'e's Conjecture) variety $Y_{\kappa/d}$ has been studied by Brou\'e and Michel \cite{BM} when $C_G(D)$ is a torus, as well as by Digne and Michel \cite{DM} when $C_G(D)$ is a Levi subgroup and $B$ is the principal block of $G$. 

\vspace{1mm}

In this setting, an $\ell$-modular system $(K,\mathcal{O},k)$ has been fixed and $\mathcal{O}$ is an extension of $\ZZ_{\ell}$ such that $K$ is large enough for $G$. Let $D$ be the defect group of the block $B$. Let $d$ be the multiplicative order of $q$ modulo $\ell$, and $\kappa$ a positive integer prime to $d$. 
The complex $H^{\bullet}(Y_{\kappa/d},K)$ of $K G$-modules (this is sometimes regarded as a graded vector space, instead of a complex with zero differential) arises from a complex - usually denoted by $R\Gamma(Y_{\kappa/d},\mathcal{O})$ - defined over $\mathcal{O}$, which produces a complex $C$ over $k$ as well, by reducing modulo $J(\mathcal{O})$. The complex $C$ is the central object of the geometric form of Brou\'e's conjecture. What we know is that it carries an action of $G$ on the right, and an action of $C_G(D)$ on the left; it is conjectured that this action can always be extended to $N_G(D)$. If this was true, then Brou\'e and Malle conjectured that, as a complex of $(kN_G(D),kG)$-bimodules, it is inducing a derived equivalence:
\begin{conjecture}\emph{\textbf{(Brou\'e, Malle - 1993,  \cite[p. 124, ($\ell$-V6)]{BrMa})}}\label{fgeom}
Let $d$ be the order of $q$ modulo $\ell$. There exists $\kappa$ such that the complex $C$ of $(kN_G(D),kG)$-bimodules described above induces a derived equivalence between $B$ and its Brauer correspondent $b$. 
\end{conjecture}
The conjectured derived equivalence induced by $C$ should also be perverse. 
As remarked in \cite[\S 5]{Dav2}, some cases of this conjecture are known, for example \cite[Theorems B, 4.13 and 4.20]{Dud} and \cite{DudRo}.

For our computational approach, the object $H^{\bullet}(Y_{\kappa/d},K)$ is too hard to manage, and so the perverse equivalence that is conjecturally induced must be obtained via a different direction. As a complex, $H^{\bullet}(Y_{\kappa/d},K)$ is predicted to fulfil the following property:
\begin{conjecture}\emph{\textbf{(Brou\'e, Malle - 1993, \cite[p.127, (d-V4)]{BrMa})}}\label{istheone}
Let $\chi$ be a unipotent character of $KG$ belonging to the principal block. The complex of cohomology $H^{\bullet}(Y_{\kappa/d},K)$ has a \textbf{unique} degree in which $\chi$ appears. This defines a 
function 
\begin{equation*}
\pi_{\kappa/d} : \Uch(B_0(G)) \to \mathbb{Z}_{\geq 0},
\end{equation*}
where $\pi_{\kappa/d}(\chi)$ is such degree.
\end{conjecture}
As we will explain in the subsection \ref{nmoved}, a map $\pi_{\kappa/d} : \Uch(B_0(G)) \to \mathbb{Z}_{\geq 0}$ can be regarded as a map $\pi_{\kappa/d} : \mathcal{S}_{B_0(H)} \to \mathbb{Z}_{\geq 0}$ using a unitriangular form of the decomposition matrix of $B_0(G)$. This is supposed to be the perversity function that characterises the conjectured perverse equivalence:
\begin{conjecture}\emph{\textbf{(Craven - Rouquier, \cite[\S 3.4.1]{Dav})}}
The derived equivalence induced by $C$ is a perverse equivalence.
The function $\pi_{\kappa/d} : \mathcal{S}_{B_0(H)} \to \mathbb{Z}_{\geq 0}$ descending from the unitriangular form of the decomposition matrix of $B_0(G)$ together with the map of Conjecture \ref{istheone} is the related perversity function.
\end{conjecture}

These conjecture are all still open and have been proved in restricted cases only. Let us try to summarise them all together and state the Geometric form of Brou\'e's Conjecture.

\begin{conjecture}\emph{\textbf{(Geometric Brou\'e's abelian defect conjecture)}}
Let $\ell$ be a prime and $G=\GG(q)$ be a finite group of Lie type, for some generic group of Lie type $\GG$ and some $p$-power $q$, where $p \neq \ell$ is a prime, and let $D \in \emph{Syl}_{\ell}(G)$. We work with an $\ell$-modular system $(K,\mathcal{O},k)$ where $\mathcal{O}$ is an extension of $\ZZ_{\ell}$ such that $K$ is large enough. There exists a Deligne-Lusztig variety $Y_{\kappa/d}$, where $d$ is the order of $q$ modulo $\ell$ and $\kappa$ is coprime to $d$, acted on by $G$ on the right and by $C_G(D)$ on the left, such that:
\begin{itemize}
\item the (right) action of $C_G(D)$ extends to an action of $H:=N_G(D)$;
\item given a unipotent character $\chi \in \Uch(B_0(G))$ of $KG$, then $\chi$ occurs in a unique cohomology group $H^{i}(Y_{\kappa/d},K)$, namely the only degree of the complex where $\chi$ appears is $i$;
\item the cohomology complex \ $C:=R\Gamma(Y_{\kappa/d},k)$, as a complex of $(B_0(G), B_0(H))$ bimodules, induces a perverse equivalence between $B_0(G)$ and $B_0(H)$;
\item the perversity function $\pi_{\kappa/d} : \mathcal{S}_{B_0(H)} \to \mathbb{Z}_{\geq 0}$ of the perverse equivalence above can be obtained by the map $\Uch(B_0(G)) \to \mathbb{Z}_{\geq 0}$ that gives the degree $i$ where $\chi \in \Uch(B_0(G))$ appears, and by a suitable bijection between $\Uch(B_0(G))$ and $\mathcal{S}_{B_0(H)}$ described in Section \ref{nmoved}.
\end{itemize}
\end{conjecture}

This conjecture gives the precise source of the perversity function providing the perversity equivalence that we rely on, but still there is no way to find it algorithmically, as we are still supposed to pass through $H^{\bullet}(Y_{\kappa/d},K)$. The decisive fact is that, conjecturally, we indeed have a relatively simple formula for $\pi_{\kappa/d}$, and this is the same formula that we will use to run our algorithm:
\begin{conjecture}\emph{\textbf{(Craven - 2012, \cite{Dav2})}}\label{conjpi}
Let $\chi \in \Uch(B_0(G))$ and let $f=f_{\chi}$ be its degree polynomial.
The perversity function from Conjecture \ref{istheone} is:
\begin{equation*}
\pi_{\kappa/d}(\chi)= \frac{\kappa}{d}(a(f)+\deg(f))+\phi_{\kappa/d}(f),
\end{equation*}
where $a(f)$ is the multiplicity of the root $q=0$, and $\phi_{\kappa/d}(f)$ is a number depending on the remaining roots of $f$.
\end{conjecture}
This conjectured result would provide a viable way to get our perversity function $\pi_{\kappa/d}$ that our algorithm strongly relies on. The list of degree polynomials related to the set of unipotent characters of a fixed block (e.g the principal) of a fixed group of Lie type can be found in the literature, for example \cite[\S 13]{Car}, or via GAP 3 \cite{gap}; finding the value of $a(f)$ and $\phi_{\kappa/d}(f)$ is also easy. It is worth remarking that some of the ground where our algorithm has taken roots is still at a conjectural level; still, there is no reason why we cannot try to use this conjectural data as an input for our algorithm, and as we will see in the following, this choice for our input has provided the expected result for $\Omega_8^{+}(2)$. 

\subsection{Bijection and perversity function}\label{nmoved} As we can see in Definition \ref{defperve}, a perverse equivalence between $B_0(G)$ and $B_0(H)$ is characterised by two data: a perversity function $\pi$ and a bijection between the sets $\mathcal{S}_{B_0(G)}$ and $\mathcal{S}_{B_0(H)}$.
From the computational point of view, both the bijection and the perversity function are inputs, and therefore finding a suitable $\pi$ together with a bijection is necessary in order to make the algorithm return a perverse equivalence. As we have explained above, there is no general formula for the perversity function related to perverse equivalences between blocks of $k G$ and $k H$, but such a formula exists at a conjectural level when $G$ is a group of Lie type. This formula is computed via the degree polynomials of the unipotent ordinary characters of $G$, and then extended to the simple $B_0(H)$-modules. An introduction to unipotent characters and the related theory can be found in \cite{Car}, together with a list of degree polynomials for some simple groups of Lie type.

Let us explain it in more detail. We assume that $G$ is a group of Lie type.
The method consists of finding bijections 
\begin{equation}\label{biggi}
\Uch(B_0(G)) \xlongrightarrow{ \ \ 1 : 1 \ \ } \mathcal{S}_{B_0(G)} \xlongrightarrow{ \ \ 1 : 1 \ \ } \mathcal{S}_{B_0(H)}
\end{equation}
between those three sets, and then defining a perversity function from a suitable function $\Uch(B_0(G)) \to \ZZ_{\geq 0}$. By an abuse of notation, we will refer to $\pi$ as the perversity function defined either on $\Uch(B_0(G))$ or $\mathcal{S}_{B_0(H)}$, under the assumption that a bijection between $\Uch(B_0(G))$ and $\mathcal{S}_{B_0(H)}$ has been fixed. In the following, we describe the two bijections appearing in (\ref{biggi}).

\begin{itemize}
\item The bijection $\Uch(B_0(G)) \xlongrightarrow{   1 : 1   } \mathcal{S}_{B_0(G)}$ is defined as follows: we order the set $\Uch(B_0(G))$ according to the given perversity function, namely $\chi \leq \chi' \iff \pi(\chi) \leq \pi(\chi')$. If two or more characters have the same value, we can arbitrarily fix an ordering for them. Permuting the rows of a decomposition matrix with respect to this order, it turns out - in the case that we will consider - that there exists a unique way to permute the list of the simple $B_0(G)$-modules (the columns) to obtain a unitriangular matrix. This unitriangular structure of the decomposition matrix gives the required bijection between $\mathcal{S}_{B_0(G)}$ and $\Uch(B_0(G))$. 

\item A more tricky part consists of finding the right bijection between $\mathcal{S}_{B_0(G)}$ and $\mathcal{S}_{B_0(H)}$. This is the bijection which is carried by the definition of perverse equivalence. In \cite{Dav2} we have a way to find the correct bijection in the case of cyclic Sylow subgroup only. In our setting, our Sylow $\ell$-subgroup is $C_{\ell} \times C_{\ell}$, however the number of modules that we consider is limited, therefore we can find the correct bijection using trial and error (the bijection will be correct if it makes the algorithm work as we will explain). Some additional numerical information will reduce the possibilities a lot; for example, the underlying perfect isometry of the derived equivalence that we aim for would imply that:
\begin{equation}\label{benny}
(-1)^{\pi(T)}\dim(T) \equiv \chi(1) \mod \ell, 
\end{equation}
where $\chi \in \Uch(B_0(G))$ and $T \in \mathcal{S}_{B_0(H)}$ correspond under the resulting bijection between $\Uch(B_0(G))$ and $\mathcal{S}_{B_0(H)}$. Therefore, if the bijection between $\Uch(B_0(G))$ and $\mathcal{S}_{B_0(G)}$ has already been obtained, the numerical information coming from the relations (\ref{benny}) restricts the possible choice for the bijection $\mathcal{S}_{B_0(G)} \xlongrightarrow{ \ \ 1 : 1 \ \ } \mathcal{S}_{B_0(H)}$.
\end{itemize}

\vspace{2mm}

Following  \cite[\S 1, \S 7]{Dav2}, we explain how the (conjecturally) valid perversity function in Conjecture \ref{conjpi} can be obtained. 

\vspace{2mm}

Let $z=r e^{i \theta}$ be a non-zero complex number and $\kappa, d$ be positive integers such that $(\kappa,d) =1$. The set $\Arg_{\kappa/d}(z)$ consists of all the positive numbers which are an argument for $z$ and are smaller than $\frac{2 \pi \kappa}{d}$, namely 
\begin{equation*} \Arg_{\kappa/d}(z) = \left\{ \theta +2 \pi h \mid h \in \ZZ, \ 0 \leq \theta +2 \pi h \leq \frac{2 \pi \kappa}{d} \right\}.
\end{equation*}
For a polynomial $f$, we denote by $\Arg_{\kappa/d}(f)$ the multiset produced by the union of all $\Arg_{\kappa/d}(z)$, where $z$ runs over all the roots of $f$ different from $0$ and $1$, counting their multiplicity. The multiplicity of $0$ as a root is denoted by $a(f)$; some authors call it the ``valuation'' of the polynomial, in other words it is the degree of the trailing term of $f$. The root $1$ is excluded as we want to count it with half its multiplicity, and we define $\phi_{\kappa/d}(f)$ as the sum of half the multiplicity of $1$ as a root of $f$ and $|\Arg_{\kappa/d}(f)|$. 

According to the Deligne-Lusztig theory (for example, see \cite{Lusz}), a group of Lie type $G$ descends from a more general object called a generic group of Lie type, often denoted by $\mathbb{G}$; this is a family of groups of Lie type parametrised by numbers of the form $q=p^s$, where $p$ is prime and $s \geq 1$ is an integer, so we can specialise $\mathbb{G}$ to the prime power $q$, and write $G=\mathbb{G}(q)$. We will not focus on the character theory of groups of Lie type in this work, and it is enough to mention that the number of unipotent characters of $G$ is actually determined at the level of $\mathbb{G}$, and in particular a unipotent character of $G$ descends from a more general object called a \emph{generic character} of $\mathbb{G}$, which depends on the type of Dynkin diagram as well as the Frobenius map acting on the diagram. To a generic unipotent character $\chi \in \Uch(\GG)$, we can associate a polynomial $f =f_{\chi} \in \QQ[x]$ such that $f(q)=\deg{\chi |_{q}}$, and $\chi |_{q}$ is the character of $G$ descending from the generic $\chi$.
Following \cite[Def. 1.2]{Dav2}, we define the  map which is conjectured to be a valid \emph{perversity function} as 
\begin{equation}\label{perversity}
\pi_{\kappa/d}(\chi |_{q}):=\frac{\kappa}{d}(a(f_{\chi})+\deg(f_{\chi}))+\phi_{\kappa/d}(f_{\chi}),
\end{equation}
where $d$ is the order of $q$ modulo $\ell$, and $\kappa$ is a positive integer coprime to $d$. In principle, it is not clear why this expression should return an integer when evaluated on unipotent characters; in \cite[Th. 1.5]{Dav2} we have a result which proves the integrality of $\pi_{\kappa/d}(\chi |_{q})$ under some conditions. Furthermore, the polynomial $f(q)$ is the product of cyclotomic polynomials and a factor of the form $a q^N$, for $N \in \ZZ_{\geq 0}$ and $a \in \QQ$, and this will make it easier to write an algorithm producing $\pi_{\kappa/d}(\chi |_{q})$ given $\chi |_{q},\kappa,d$.

\begin{example}
Let $\chi$ be the generic trivial character of $\GG$, and hence $f_{\chi}=1$. By the expression (\ref{perversity}) we have $\pi_{\kappa/d}(\chi |_{q})=0$, since $a(f_{\chi})=\deg(f_{\chi})=\phi_{\kappa/d}(f_{\chi})=0$.
\end{example}

\section{Brou\'e's Conjecture for $\Omega^{+}_8(2)$}

We are going to apply Proposition \ref{forestanera} to prove Theorem \ref{mine}.
\subsection{The simple group $G:=\Omega^{+}_8(2)$} We have $|G|=2^{12} \cdot 3^5 \cdot 5^2 \cdot 7$. When considering principal blocks, Conjecture \ref{splee} has to be checked for the case $\ell=5$ only, since both the $2$-Sylow and $3$-Sylow subgroups of $G$ are not abelian, and the case $\ell=7$ is already known since the defect group is cyclic. 
Therefore we set $\ell:=5$, $D \in \Syl_{\ell}(G)$ and $H:=N_G(D)$. We have $D \simeq C_{\ell} \times C_{\ell}$, therefore the algorithmic strategy that we outlined above applies for the principal $5$-block $B_0(G)$. 

\vspace{2mm}

There are three conjugacy classes of subgroups of $H$ of order $5$, represented by $Q_1,Q_2,Q_3$. For each of those, we have $\bar{C}_G(Q_i) \simeq A_5$ and $\bar{C}_H(Q_i) \simeq D_{10}$, the alternating group and the dihedral group of order $10$ respectively, and the subgroups $\bar{C}_G(Q_i)$ and $\bar{C}_H(Q_i)$ have been defined in Section \ref{decc}. This information will be necessary to determine the stable equivalence between $B_0(G)$ and $B_0(H)$.

Before reporting the result of the algorithms, we summarise the main information in the following table: perversity function, bijection between $\mathcal{S}_{B_0(G)}$ and $\mathcal{S}_{B_0(H)}$, unipotent characters. Each $C_i$ denotes the Green correspondent of $S_i$ down to $H$, and its dimension is reported.

\begin{center}
\begin{tabular}{| c | c | c| c| c| c|} 
 \hline
$\pi_{1/4}$ & $\chi$ & Polynomial& $kH$-mod & $B_0(G)$-mod & dim $C_i$ \\
 \hline
 0 & $1_1$ & 1& $T_{1}=1_{1}$     &    $S_{1}=1_1$    &      $\dim(C_1)=1$   \\  
 3 & $84_1$& $q^2 \Phi_3(q) \Phi_6(q)$ &$T_{2}=1_{2}$     &    $S_{2}=83_1$    &     $ \dim(C_2)=8$   \\ 
 3 & $84_2$& $q^2 \Phi_3(q)\Phi_6(q)$ &$T_{3}=1_{3}$        &   $S_{3}=83_2$  &       $ \dim(C_3)=8$  \\ 
 3 & $84_3$&  $q^2 \Phi_3(q)\Phi_6(q)$ & $T_{4}=1_{4}$      &  $S_{4}=83_3$ &           $ \dim(C_4)=8$   \\ 
 4 & $972_1$ & $ q^3\Phi_2(q)^4\Phi_6(q)/2$ & $T_{9}=2_{1}$      & $S_{9}=722_1$  &        $ \dim(C_9)=47$    \\ 
 5 & $28_1$ & $q^3\Phi_1(q)^4\Phi_3(q)/2$  & $T_{10}=2_{2}$   & $S_{10}=28_1$ &      $ \dim(C_{10})=28$    \\
 5 & $1344_1$ & $q^6 \Phi_3(q)\Phi_6(q)$ & $T_{5}=1_{5}$   & $S_{5}=539_1$  &          $ \dim(C_5)=64$    \\ 
 5 & $1344_2$ & $q^6 \Phi_3(q)\Phi_6(q)$ & $T_{6}=1_{6}$   & $S_{6}=539_2$  &              $ \dim(C_6)=64$  \\
 5 & $1344_3$ & $q^6 \Phi_3(q)\Phi_6(q)$ & $T_{7}=1_{7}$   & $S_{7}=539_3$  &               $ \dim(C_7)=64$ \\
 6 & $4096_1$ & $q^{12}$& $T_{8}=1_{8}$     & $S_{8}=1729_1$  &         $ \dim(C_8)=29$   \\ 
 \hline
\end{tabular}
\end{center}

\subsection{Representation theory of $G$}

The finite field $\mathbb{F}_5$ is a splitting field for $G$; in fact, it is easily checked on MAGMA that every simple $kG$-module can be realised over $\mathbb{F}_5$, and such two properties are equivalent (see, for instance, \cite[Th. 1.14.8, \S 1]{Link}). Therefore, we can carry out all our computations over $\mathbb{F}_5$. 

\vspace{1mm}

We have $10$ simple $B_0(G)$-modules $S_1, \dots, S_{10}$. Following the notation of \cite{Atlas}, we set $S_1 := k$, and
\begin{center}
$\begin{array}{c}
S_2=83_1\\
S_3=83_2\\
S_4=83_3\\
\end{array} \ \begin{array}{c}
S_5=539_1\\
S_6=539_2\\
S_7=539_3\\
\end{array} \ \begin{array}{c}
S_8=1729_1\\
S_9=722_1\\
S_{10}=28_1.\\
\end{array}$
\end{center}

Via MAGMA, we can find that $\text{Out}(G)  \cong S_3$, and it permutes the three modules of order $539$ and the three modules of order $83$. 
The Modular Atlas \cite{Atlas} provides a decomposition matrix of $G$ in characteristic $5$. This can be easily re-arranged in a unitriangular shape as follows (the non-unipotent characters are not reported here), and we also include the value of the perversity function that we will run our algorithm \texttt{PerverseEq} with:

\vspace{3mm}

\begin{equation*}
\begin{tabular}{ |c|c | c c c c c c c c c  c  |} 
 \hline
 \multicolumn{12}{|c|}{$B_0(G)$, $G=\Omega^{+}_8(2)$, $\ell=5$} \\
 \hline
$\pi_{1/4}$ & Unipotent Char & $S_1$  & $S_2$ & $S_3$ & $S_4$ & $S_9$ &$S_{10}$ &$S_5$ & $S_6$ &$S_7$ & $S_8$ \\
\hline
 0 & $1_1$ & 1  &   &  & & & & & &  & \\  
 3 & $84_1$  &1  & 1 &  & & & & & &  &\\ 
 3 & $84_2$   & 1  &  & 1 & & & & & &  & \\ 
 3 & $84_3$  & 1  &  &  &1 & & & & &  & \\ 
 4 & $972_1$  & 1  & 1 & 1 &1 &1 & & & &  & \\ 
 5 & $28_1$  &  &  &  & & & 1 & & &  & \\ 
 5 & $1344_1$   & &1  &  & & 1& &1 & &  & \\ 
 5 & $1344_2$   & &  & 1 & & 1& & &1 &  & \\ 
 5 & $1344_3$   & &  & & 1& 1& & & &1  & \\ 
 6 & $4096_1$ & &  &  & &1 &1 &1 &1 & 1 &1 \\ 
 \hline
\end{tabular}
\end{equation*}

\vspace{2mm}

\subsection{Representation theory of $H$}
Via MAGMA, we find $|H|=400$ and in particular $H \cong D \rtimes S$, where $S$ is the complex reflection group:
\begin{center}
$S \cong G(4,2,2) \cong \{s, t, u | s^2 = t^2 = u^2 = 1,  stu = tus = ust \}$. 
\end{center}
The group algebra $k H$ decomposes in one block only; it is a general fact that if a finite group $H$ contains a normal $\ell$-subgroup $R$ such that $C_H(R) \leq R$, then $kH$ has exactly one block. This is a consequence of \cite[Th.4, \S 15]{Alp} and of $R$ being contained in any defect group of any block, by \cite[Th.6, \S 13]{Alp} . In our case, this happens by choosing $R=D$. 

The algebra $k H=B_0(H)$ has $10$ simple modules, all absolutely simple, eight of them of dimension $1$ and two of dimension $2$. 
According to our labelling, $T_1, \dots, T_8$ have dimension $1$ and $T_9,T_{10}$ have dimension $2$; $T_1$ denotes the trivial module. 
When writing the socle structure of a module, $T_i$ is abbreviated to $i$. The labelling that we choose for the set of simple $B_0(H)$-modules is such that the socle series of the projective cover of the trivial module $T_1$ is:
\begin{equation*}
\mathcal{P}(1)=\begin{array}{c}
\ 1  \\
\ 10  \\
\ 2  \ 3  \ 4  \\
\ 9  \ 9  \\
\ 1  \ 1  \ 5  \ 6  \ 7  \\
\ 10  \ 10  \\
\ 2  \ 3  \ 4  \\
\ 9  \\
\ 1  \\
\end{array}
\end{equation*}
The three modules $T_2, T_3, T_4$ are permuted by $\text{Out}(H) \cong S_3$, and the same happens for $T_5,T_6,T_7$. We stipulate that $T_5:=T_4 \otimes T_3$,  $T_6:=T_2 \otimes T_4$,  $T_7:=T_3 \otimes T_2$; therefore, once we have distinguished $T_2,T_3,T_4$, we have distinguished $T_5,T_6,T_7$ as well. It remains to identify $T_8$, and this is the exterior square of $T_9$.

The three conjugacy classes of subgroups of order $5$ can be labelled by $Q_1,Q_2,Q_3$, where each $Q_i$ is a representative of each class.  A concrete mean of fixing this labelling is looking at the output of our algorithm \texttt{PerverseEq}; in the next chapter we will introduce $10$ complexes $X_i$ of $B_0(H)$-modules resulting from the application of \texttt{PerverseEq}. Looking at the complexes $X_2, X_3, X_4$ in degree $-1$, we see that three modules $R_1, R_2, R_3$ of dimension $10$ appear. We define $Q_i$ as a vertex of $R_i$ for $i=1,2,3$. This distinguishes each $Q_i$ from the other two. The structure of each $R_i$ is:

\begin{center} 
$R_{1}=\begin{array}{c}
\ 10  \\
\ 3  \ 4  \\
\ 9  \\
\ 1  \ 5  \\
\ 10  \\
\end{array}$ \ \ \
$R_{2}=\begin{array}{c}
\ 10  \\
\ 2  \ 4  \\
\ 9  \\
\ 1  \ 6  \\
\ 10  \\
\end{array}$ \ \ \
$R_{3}=\begin{array}{c}
\ 10  \\
\ 2  \ 3  \\
\ 9  \\
\ 1  \ 7  \\
\ 10  \\ 
\end{array}$ \ \ \ 
\end{center}

\vspace{3mm}

\subsection{Stable equivalence}

We perform the construction that we have described in Section \ref{decc} of the complexes $C \otimes_{kG} S$; we will actually build the term of degree $-1$, namely $T_Q \otimes_{k \bar{C}_G(Q)} S_{N_G(Q)}$. 
As we remarked at the beginning of the section, we have three conjugacy classes of subgroups of order $5$ in $H$. We can denote by $Q$ a generic subgroup of order $5$; the result from the construction of the stable equivalence is the same for each of those three, up to isomorphism. 

\vspace{2mm}

We recall the notation from Section \ref{decc} and from \cite[\S 3.3.1]{Dav}: $\bar{N}_G(Q)$ and $\bar{N}_H(Q)$ are complements of $Q$ inside $N_G(Q)$ and $N_H(Q)$ respectively, and they can be chosen such that $\bar{N}_H(Q) \leq \bar{N}_G(Q)$. We need $Q$-complements of centralisers as well, and then we take $\bar{C}_G(Q)= C_G(Q) \cap \bar{N}_G(Q)$ and $\bar{C}_H(Q)= C_H(Q) \cap \bar{N}_H(Q)$.

For each of the three $Q=Q_1,Q_2,Q_3$, we have $\bar{C}_H(Q) \cong D_{10}$, the dihedral group of order $10$, and $\bar{C}_G(Q) \cong A_5$, the alternating group of order $60$. Using MAGMA, we can check that as a $k [\bar{C}_H(Q) \times \bar{C}_G(Q)^{\text{opp}}]$-module, we have that $k\bar{C}_G(Q) \cong M_Q \oplus V$, where $M_Q$ and $V$ are indecomposable, $\dim(M_Q)=35$, $\dim(V)=25$, and only $M_Q$ lies in the principal block. So we have 
\begin{equation*}
e_{\bar{C}_H(Q)} k\bar{C}_G(Q)e_{\bar{C}_G(Q)} \cong M_Q.
\end{equation*}
In particular, no projective summand appears in this decomposition. As we have $\bar{N}_H(Q) \leq N_G(Q)$ and in this case $\bar{C}_G(Q) \mathrel{\unlhd} N_G(Q)$, then $\bar{N}_H(Q)$ normalises $\bar{C}_H(Q)$ and the action of $\bar{C}_H(Q) \times \bar{C}_G(Q)^{\text{opp}}$ on $k\bar{C}_G(Q)$ can be extended to a natural action of $N_{\Delta}= (\bar{C}_H(Q) \times \bar{C}_G(Q)^{\text{opp}}) \Delta \bar{N}_H(Q)$; it turns out that, as a $k N_{\Delta}$-module, $k\bar{C}_G(Q)$ does not decompose any further than $M_Q$ and $V$. So we conclude that 
\begin{equation}\label{projdec}
e_{\bar{C}_H(Q)} k\bar{C}_G(Q) e_{\bar{C}_G(Q)} \cong M_Q
\end{equation}
as a $k N_{\Delta}$-module.

\vspace{2mm}
The representation theory of $k\bar{C}_H(Q)$ and $k\bar{C}_G(Q)$ is briefly recalled: they decompose into one and two blocks respectively, and $k \bar{C}_H(Q)$ has two simple modules $1_1$, $1_2$, and $k\bar{C}_G(Q)$ has three simple modules $1_1,3_1,5_1$, where the first two of them belong to the principal block. Each simple module can be seen as a simple module for $k [\bar{C}_H(Q) \times \bar{C}_G(Q)^{\text{opp}}]$, where the original group acts as usual and the other factor acts trivially. The set of simple modules for $k [\bar{C}_H(Q) \times \bar{C}_G(Q)^{\text{opp}}]$ is indeed $1_1 \otimes 1_1, 1_2 \otimes 1_1, 1_1 \otimes 3_1,1_2 \otimes 3_1,1_1 \otimes 5_1,1_2 \otimes 5_1.$
The Brauer tree of the principal block of $k\bar{C}_G(Q)$ is:

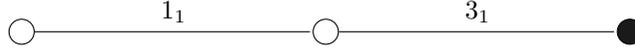
\vspace{2mm}

\begin{center}
\begin{tikzpicture}[shorten >=1pt, auto, node distance=0.5cm,
   node_style/.style={circle,draw=black,font=\sffamily\Large\bfseries},
   node_stylee/.style={circle,draw=black,fill=black!90!,font=\sffamily\Large\bfseries},
   edge_style/.style={draw=black}]

    \node[node_style] (v1) at (-2,2) {};
    \node[node_style] (v2) at (2,2) {};
    \node[node_stylee] (v3) at (6,2) {};
    \draw[edge_style]  (v1) edge node{$1_1$} (v2);
    \draw[edge_style]  (v2) edge node{$3_1$} (v3);

    \end{tikzpicture}

\vspace{2mm}

\end{center}
We recall that the map $\gamma : \mathcal{S}_{B_0(\bar{C}_G(Q))} \to  \mathcal{S}_{B_0(\bar{C}_H(Q))}$ described by the relation \ref{penn} is used to determine a projective cover of $M_Q$. Via MAGMA, we can find that 
$\gamma$ sends the trivial module to the trivial module and $3_1$ to $1_2$.
Therefore, as expected according to Section \ref{decc}, our computation in MAGMA confirms that a projective cover of $M_Q$ is of the form 
\begin{equation*}
\mathcal{P}(1_1 \otimes 1_1) \oplus \mathcal{P}(1_2 \otimes 3_1) \twoheadrightarrow M_Q. 
\end{equation*}
The subset $\mathcal{E}$  of $\mathcal{S}_{B_0(\bar{C}_G(Q))}$ is defined by looking at the Brauer tree of $B_0(\bar{C}_G(Q))$: the distance $d$ between the exceptional vertex and the edge of the trivial module is $1$; so the subset $\mathcal{E} \subset \mathcal{S}_{B_0(\bar{C}_G(Q))}$ such that the distance of the edge from the exceptional vertex is $1+1 = 0 \ (\text{mod} \ 2)$ is formed of $3_1$ only. 

We can now run the algorithm \texttt{FinalStabEq}; this would compute the image of each element in $\mathcal{S}_{B_0(G)}$ under the stable equivalence $L$ in Proposition \ref{forestanera}; if the result matches with the output of the algorithm \texttt{PerverseEq}, then by Proposition \ref{forestanera} we have a splendid derived equivalence between $B_0(G)$ and $B_0(H)$.

As no projective summand appears in the decomposition \ref{projdec}, we deduce that $T_Q=U_Q=\mathcal{P}(1_2 \otimes 3_1)$, so we have to compute $T_Q \otimes_{k\bar{C}_G(Q)} S_{N_G(Q)}$ for every $Q=Q_1,Q_2,Q_3$ and for every simple $B_0(G)$-module $S=S_1, \dots, S_{10}$. 
Our computations show that:
\begin{equation*}
\begin{split}
S&=S_1, \ \ \  \bigoplus_{Q=Q_1,Q_2,Q_3} T_Q \otimes_{k\bar{C}_G(Q)} S_{N_G(Q)} \cong \{0 \}; \\
S&=S_2, \ \ \ \bigoplus_{Q=Q_1,Q_2,Q_3} T_Q \otimes_{k\bar{C}_G(Q)} S_{N_G(Q)} \cong R_1 \oplus \{0 \} \oplus \{0\} \cong R_1; \\
S&=S_{3}, \ \ \ \bigoplus_{Q=Q_1,Q_2,Q_3} T_Q \otimes_{k\bar{C}_G(Q)} S_{N_G(Q)} \cong  \{0 \} \oplus R_2 \oplus \{0\} \cong R_2; \\
S&=S_4, \ \ \ \bigoplus_{Q=Q_1,Q_2,Q_3} T_Q \otimes_{k\bar{C}_G(Q)} S_{N_G(Q)} \cong \{0\} \oplus  \{0 \} \oplus R_3 \cong R_3;\\
S&=S_5, \ \ \ \bigoplus_{Q=Q_1,Q_2,Q_3} T_Q \otimes_{k\bar{C}_G(Q)} S_{N_G(Q)} \cong R_1 \oplus U_2 \oplus U_3; \\
S&=S_6, \ \ \ \bigoplus_{Q=Q_1,Q_2,Q_3} T_Q \otimes_{k\bar{C}_G(Q)} S_{N_G(Q)} \cong U_1 \oplus R_2 \oplus U_3; \\
S&=S_7, \ \ \ \bigoplus_{Q=Q_1,Q_2,Q_3} T_Q \otimes_{k\bar{C}_G(Q)} S_{N_G(Q)} \cong U_1 \oplus U_2 \oplus R_3; \\
S&=S_8, \ \ \ \bigoplus_{Q=Q_1,Q_2,Q_3} T_Q \otimes_{k\bar{C}_G(Q)} S_{N_G(Q)} \cong U^{*}_1 \oplus U^{*}_2 \oplus U^{*}_3; \\
S&=S_9, \ \ \ \bigoplus_{Q=Q_1,Q_2,Q_3} T_Q \otimes_{k\bar{C}_G(Q)} S_{N_G(Q)} \cong U^{*}_1 \oplus U^{*}_2 \oplus U^{*}_3; \\
S&=S_{10}, \ \ \ \bigoplus_{Q=Q_1,Q_2,Q_3} T_Q \otimes_{k\bar{C}_G(Q)} S_{N_G(Q)} \cong R^{*}_1 \oplus R^{*}_2 \oplus R^{*}_3.
\end{split}
\end{equation*}
where all these isomorphisms are in the stable category, namely up to projective summands. Here, $U_i$ are indecomposable $B_0(H)$-modules of dimension $90$ with vertex $Q_i$ and with $3$-dimensional source. Their structures are:
\begin{center}
$U_1=\begin{array}{c}
\ 6  \ 7  \\
\ 10  \ 10  \\
\ 2  \ 2  \ 3  \ 4  \ 8  \ 8  \\
\ 9  \ 9  \ 9  \\
\ 1  \ 5  \ 6  \ 6  \ 7  \ 7  \\
\ 10  \ 10  \\
\ 2  \ 8  \\
\end{array}$ \ \ \ \
$U_2=\begin{array}{c}
\ 5  \ 7  \\
\ 10  \ 10  \\
\ 2  \ 3  \ 3  \ 4  \ 8  \ 8  \\
\ 9  \ 9  \ 9  \\
\ 1  \ 5  \ 5  \ 6  \ 7  \ 7  \\
\ 10  \ 10  \\
\ 3  \ 8  \\
\end{array}$ \ \ \ \
$U_3=\begin{array}{c}
\ 5  \ 6  \\
\ 10  \ 10  \\
\ 2  \ 3  \ 4  \ 4  \ 8  \ 8  \\
\ 9  \ 9  \ 9  \\
\ 1  \ 5  \ 5  \ 6  \ 6  \ 7  \\
\ 10  \ 10  \\
\ 4  \ 8  \\
\end{array}$ \ \ \ \
\end{center}

It remains to run the algorithm \texttt{PerverseEq} in order to compare this result and possibly apply Proposition \ref{forestanera}.

\subsection{Perverse equivalence} In the following, each complex $X_i$ will denote the algorithmic output corresponding to the simple $B_0(H)$-module $T_i$, for each $i=1,\dots, 10$. As $\pi(T_1)=0$, the complex $X_1$ is just: 
\begin{center}
$X_1: 0 \to T_1  \to 0.$
\end{center}
The next value of $\pi$ is $3$, with $T_2,T_3$ and $T_4$. The algorithm constructs:
\vspace{2mm}
\begin{center}
$X_2: 0 \to \mathcal{P}(2) \to \mathcal{P}(10) \to \mathcal{P}(5)  \oplus R_{1} \twoheadrightarrow C_2 \to 0,$

$X_3: 0 \to \mathcal{P}(3) \to \mathcal{P}(10) \to \mathcal{P}(6)  \oplus R_{2} \twoheadrightarrow C_3 \to 0$, \\

$X_4: 0 \to \mathcal{P}(4) \to \mathcal{P}(10) \to \mathcal{P}(7)  \oplus R_{3} \twoheadrightarrow C_4 \to 0$. \\
\end{center}
In the following, we get:
\begin{equation*}
\begin{split}
X_9: \  & \mathcal{P}(9) \to \mathcal{P}(8) \oplus \mathcal{P}(10) \oplus \mathcal{P}(10) \to  \\  & \to \mathcal{P}(2)\oplus \mathcal{P}(3) \oplus \mathcal{P}(4) \oplus \mathcal{P}(5) \oplus \mathcal{P}(6) \oplus \mathcal{P}(7) \oplus \mathcal{P}(8) 
  \to  \\  & \to  \mathcal{P}(10) \oplus U^{*}_1 \oplus U^{*}_2 \oplus U^{*}_3 \twoheadrightarrow C_9 \to 0.
\end{split}
\end{equation*}
From $X_2$,$X_3$,$X_4$ we can see how $5,6,7$ are permuted. Now we move to the triple $T_5$, $T_6$, $T_7$. 
\begin{equation*}
\begin{split}
X_5: \  & \mathcal{P}(5) \to   \mathcal{P}(8) \oplus \mathcal{P}(10) \to \mathcal{P}(6)\oplus \mathcal{P}(7)\oplus \mathcal{P}(9)  \to \\ & \to \mathcal{P}(5)\oplus \mathcal{P}(6)\oplus \mathcal{P}(7)\oplus \mathcal{P}(10) \to \\
& \to   \mathcal{P}(8) \oplus \mathcal{P}(9) \oplus R_{1} \oplus U_2 \oplus U_3 \twoheadrightarrow C_5 \to 0.
\end{split}
\end{equation*}
\begin{equation*}
\begin{split}
X_6: \ &  \mathcal{P}(6) \to  \mathcal{P}(8) \oplus \mathcal{P}(10) \to \mathcal{P}(5)\oplus \mathcal{P}(7)\oplus \mathcal{P}(9)  \to \\ & \to \mathcal{P}(5)\oplus \mathcal{P}(6)\oplus \mathcal{P}(7)\oplus \mathcal{P}(10) \to \\
&  \to   \mathcal{P}(8) \oplus \mathcal{P}(9) \oplus R_{2} \oplus U_1 \oplus U_3 \twoheadrightarrow C_6 \to 0.
\end{split}
\end{equation*}

\begin{equation*}
\begin{split}
X_7: \ & \mathcal{P}(7)  \to  \mathcal{P}(8) \oplus \mathcal{P}(10) \to \mathcal{P}(5)\oplus \mathcal{P}(6)\oplus \mathcal{P}(9)  \to \\ & \to \mathcal{P}(5)\oplus \mathcal{P}(6)\oplus \mathcal{P}(7)\oplus \mathcal{P}(10) \to \\
&  \to   \mathcal{P}(8) \oplus \mathcal{P}(9) \oplus R_{3} \oplus U_1 \oplus U_2 \twoheadrightarrow C_7 \to 0.
\end{split}
\end{equation*}
Finally:
\begin{equation*}
\begin{split}
X_{10}: \  \mathcal{P}(10) & \to  \mathcal{P}(5) \oplus \mathcal{P}(6) \oplus \mathcal{P}(7) \to \mathcal{P}(5)\oplus \mathcal{P}(6)\oplus \mathcal{P}(7)\oplus \mathcal{P}(8) \to \\
& \to \mathcal{P}(8) \oplus \mathcal{P}(9) \oplus \mathcal{P}(9) \to \mathcal{P}(10) \oplus R^{*}_{1} \oplus R^{*}_{2} \oplus R^{*}_{3} \twoheadrightarrow C_{10} \to 0.
\end{split}
\end{equation*}

\begin{equation*}
\begin{split}
X_8: \  \mathcal{P}(8) & \to  \mathcal{P}(8) \oplus \mathcal{P}(8) \to \mathcal{P}(5)\oplus \mathcal{P}(6)\oplus \mathcal{P}(7)  \to \mathcal{P}(10)\oplus \mathcal{P}(10) \to \\
&  \to   \mathcal{P}(8) \oplus \mathcal{P}(9) \oplus \mathcal{P}(9) \to U^{*}_1 \oplus U^{*}_2 \oplus U^{*}_3 \twoheadrightarrow C_8 \to 0.
\end{split}
\end{equation*}
In the following, we write the table of cohomology that could be used to reconstruct the uni-triangular structure of the decomposition matrix, as we described at the end of Section \ref{atee}.
For compactness, we have set $A:=2/3/4/9/9/10/5/6/7/8$ (see complex $X_8$).
\begin{equation}\label{scle}
\begin{tabular}{| c| c| c c c c c c |c|}
 \hline
$X_i$ & $\pi_{1/4}$ &  $H^{-6}$ &  $H^{-5}$ &  $H^{-4}$ &  $H^{-3}$ &  $H^{-2}$ &  $H^{-1}$ & Total \\
 \hline
 $X_2$ & 3 &     &      &    & 2&  & 1 & 2-1     \\  
 $X_3$ &  3 &   &        &    & 3& &1 & 3-1   \\ 
  $X_4$&   3 &     &   &     & 4& & 1&  4-1  \\ 
$X_9$ &  4&       &   &    2/3/4/9 & & $1 \oplus 1$ & &  9-4-3-2+1+1  \\ 
$X_{10}$ &  5 &      &   1/10&     1& & & &  10  \\ 
$X_5$ &   5 &     &  3/4/9/5 &     &1 & & &  5-9+3+4-1  \\ 
$X_6$ &   5 &     &   2/4/9/6&     &1 & & &   6-9+2+4-1 \\ 
$X_7$ &   5  &    &  2/3/9/7 &     & 1& & &    7-9+2+3-1\\ 
$X_{8}$ &   6 & $A$  &   & 1    & & & & 1-2-3-4-5-6-7+8+9+9-10  \\ 
 \hline
\end{tabular}
\end{equation}

\vspace{3mm}

We can finally conclude that: 
\begin{theorem}
The principal blocks $B_0(G)$ and $B_0(H)$ are splendidly derived equivalent. 
\end{theorem}
\begin{proof}
By Proposition \ref{forestanera}, it is enough to check that for each $S=S_i$, the following isomorphisms in the stable category hold:

\begin{equation*} 
C \otimes_{kG} S_{i} \cong X_i \ \ \ \text{in} \ \ \ \ \underline{\text{mod}}(kH).
\end{equation*}
In degree $0$ we have the Green correspondent of $S$, and up to degree $-2$ the terms of $X_i$ are projective; it remains to check the degree $-1$, namely:
\begin{equation*}
\bigoplus_{Q=Q_1,Q_2,Q_3} T_Q \otimes_{k\bar{C}_G(Q)} S_{N_G(Q)} \cong X_i^{-1} \ \ \ \text{in} \ \ \ \ \underline{\text{mod}}(kH),
\end{equation*}
and this holds by comparing the result of the algorithm above. 
\end{proof}

\section{Appendix A: the algorithm \texttt{FinalStabEq}} 

This algorithm aims to implement the construction of the stable equivalence described in \cite{Dav}. What we will actually build are the \emph{images} of the simple $B_0(G)$-modules $\mathcal{S}_{B_0(G)}$ under this stable equivalence; the algorithm is then meant to return the complexes of $B_0(H)$-modules described in \cite{Dav}. We recall the notation of \cite{Dav} that we have already introduced in Section \ref{decc}: we have a $kN_{\Delta}$-module $T_Q$ and a $kN_G(Q)$-module $S$; previously, $S$ denoted a $kG$-module, but as we need to restrict it to $N_G(Q)$ even before running the algorithm, we can directly regard it as an $kN_G(Q)$-module.
The tensor product $T_Q \otimes_{k\bar{C}_G(Q)} S$ has a natural structure of $N_{\Delta} \times N_G(Q)$-module, where $N_G(Q)$ acts trivially on $T_Q$ and  $N_{\Delta}$ acts trivially on $S$. Our construction involves $T_Q \otimes_{k\bar{C}_G(Q)} S$ as a $N_H(Q)$-module; this means that we consider the copy of $N_H(Q)$ embedded inside $N_{\Delta} \times N_G(Q)$ as described in \cite{Dav}, take $(T_Q \otimes S)_{N_H(Q)}$ and build the quotient $(T_Q \otimes S)_{N_H(Q)}/\langle R \rangle_{N_H(Q)} $, where $R=\{ct \otimes s - t \otimes cs \mid c \in \bar{C}_{G}(Q), t \in T_Q, s \in S \}$; here, the action of $\bar{C}_{G}(Q)$ on $T_Q$ is meant to be carried by the copy of $\bar{C}_{G}(Q)$ inside $N_{\Delta}$, and as for $S$ we have the action of $\bar{C}_{G}(Q)$ lying inside $N_G(Q)$. With an abuse of notation, we are implicitly using that $\bar{C}_{G}(Q)$ is fixed at the beginning as a subgroup of $G$, and then the expression $ct \otimes s - t \otimes cs$ is clear. 
The main difficulty of this algorithm is about how to build the set $R$. First, we notice that as we consider the $N_H(Q)$-span, we do not really need to construct each vector of the shape $ct \otimes s - t \otimes cs$, but we can restrict $t$ to the elements of a basis of $T_Q$, $s$ to the elements of a basis of $S$, and $c$ to a set of generators of $\bar{C}_G(Q)$, typically a set of two generators. However, as some $S$ have dimension in the thousands, the tensor product $T_Q \otimes S$ would have a prohibitive dimension, but we can skip this problem by remarking two facts:
\begin{enumerate}
\item $S$ is the restriction of a simple $kG$-module down to $N_G(Q)$; then, it is in general decomposable, and it will split in a number of indecomposable non-projective and projective summands: $S=S'_1 \oplus \dots \oplus S'_r \oplus P_1 \oplus  \cdots \oplus P_e$, where $\{ S'_i \}_{i=1}^{r}$ are non-projective and $\{ P_j \}_{j=1}^{e}$ are projective. Decomposing a module of dimension in the thousands can be hard, but in general it is easy to detect and delete (by quotienting out) all the projective summands - as a projective submodule is a summand - and end up with the non-projective part of $S$ only, which is in general very small. 
As the tensor product over a subalgebra is linear, we have:
\begin{equation*}
\ \ \ \ \ \ \ \ \ \ T_Q \otimes_{k\bar{C}_G(Q)} S= \left( \bigoplus_{i=1}^{r} T_Q  \otimes_{k\bar{C}_G(Q)} S'_i \right) \oplus  \left( \bigoplus_{j=1}^{e} T_Q  \otimes_{k\bar{C}_G(Q)} P_j \right).
\end{equation*}
This shows that we can focus on indecomposable modules $S$ only. A further decomposition can be carried when the module $T_Q$ is not indecomposable. Moreover, we realise that it is convenient to compute $T_Q  \otimes_{k\bar{C}_G(Q)} P_j$ at the beginning once and for all, so the contribution of the projective part of $S$ to $T_Q \otimes_{k\bar{C}_G(Q)} S$ is immediately known as soon as we have the decomposition of $S$. 
\item  Now we have to find $T_Q \otimes_{k\bar{C}_G(Q)} S'$, where $S'$ is indecomposable. The summands $S'$ of $S$ will often be small enough to proceed with the direct computation, but sometimes not. Although $S'$ is now indecomposable, we notice that in order to get vectors $ct \otimes s - t \otimes cs, \ t \in T_Q, s \in S'$, we only care about the action of $\bar{C}_G(Q)$. So in a computational setting, we can restrict both $T_Q$ and $S'$ further down to $\bar{C}_{G}(Q)$. For example, if the decomposition of $T_Q$ as a $k\bar{C}_G(Q)$-module is $(T_Q)_{\bar{C}_G(Q)}=T_1 \oplus \dots \oplus T_m$ for some $m \geq 1$, then a basis of $T_Q$ as a vector space can be chosen as the union of basis for each subspace $T_1, \dots, T_m$; the massive computational advantage is that an arbitrary element $t$ of the basis of $T_Q$ can now be seen as a vector of some $T_j$, for $j=1,\dots, m$, which is remarkably smaller and so the matrix-vector multiplications $t \cdot c$ is done almost immediately in each case that we considered. As a vector in $T_j$, then $t\cdot c$ can be easily coerced inside $T_Q$ and tensored with $s$; the same argument applies to the $\bar{C}_G(Q)$-summands of $S'$.
\end{enumerate}
This method allows us to build the term in degree $-1$ which is supposed to come out from the image of the simple $B_0(G)$-modules under our stable equivalence. The algorithm is mostly based on three parts. First of all, for a given $kN_G(Q)$-module $S$, we detect all the indecomposable summands and their multiplicities - as using the command \texttt{IndecomposableSummands()} is not the best option when $S$ has dimension in some thousands. Given $S$ and the list of indecomposable projective $kN_G(Q)$-modules, the following returns a list recording how many times each projective appears as a summand of $S$, and a module being a copy of $S$ without its projective summands. In the following algorithm, we make use of \texttt{RemFree}, that we have not copied here; this take a module $M$, a positive number $n$, and for $n$ times it tries to generate a free submodule in $M$ to quotient by. If $n$ is large enough, it quotients $M$ by enough free summand (a free submodule is a summand), we ultimately get the non-free part of $M$ as an output. 

\begin{verbatim}
function SplitL(M,LP);
/* How many times should we try to look for free summands? 
The potential number is Dim(M) div #Group(M), the greatest integer 
less than or equal to Dim(M)/#G. As RemFree can fail, 
we will check two times this number.  */
nf:=Dimension(M) div #Group(M);
if not (nf eq 0) then 
   T:=RemFree(M,2*nf);
   else T:=M;
end if;
/* c tells me how many free summands we have removed from M */
c:=(Dimension(M)-Dimension(T)) div #Group(M);
/* LN is a list of integers. It will track how many times each 
projective is found inside M, and will be returned in the end. */
if c eq 0 then 
   LN:=[0 : x in LP];
   else LN:=[c*Dimension(Socle(x)) : x in LP];
end if;
/* Now we focus on T, to find the remaining projective summands */
for k in [1..#LP] do   
    B,n:=CountProj(T,LP[k]); delete T; T:=B; delete B;
    LN[k]:=LN[k]+n; delete n; 
end for;
return T,LN;
end function;
\end{verbatim}

Given the finite group $G$, the $\ell$-local subgroup $H$ (which will always be the normaliser of a Sylow $\ell$-subgroup), a cyclic group $Q$ of order $\ell$ contained in $H$ and its normaliser $N_G(Q)$ - that we denote in the code as NG - the following \texttt{StableEqSetup} returns the $kN_{\Delta}$-module $V=k \bar{C}_G(Q)$, which will provide, as it is described in \cite{Dav}, our module $T_Q$. Moreover, the code returns the groups denoted as BCG, IBCG, NH, BNH, IBNH, IBCH; they are, respectively, a copy of $\bar{C}_G(Q)$ in $N_G(Q)$, a copy of $\bar{C}_G(Q)$ in $N_{\Delta}$, a copy of $N_H(Q)$ and $\bar{N}_H(Q)$ inside $N_G(Q)$, a copy of $\bar{N}_H(Q)$ inside $N_{\Delta}$, and a copy of $\bar{C}_G(Q)$ inside $N_{\Delta}$. We do not need that the code returns the group $N_{\Delta}$ as well, as it is already carried by $V$, and it is easily recovered by using the command \texttt{Group()}. Each of these group is returned as generated by two elements (we assume that this is always possible), and for this purpose we use the short function \texttt{GenTwoEl}. The sole reason why we prefer to turn the set of generators of such groups into a set of two elements is based on shorter calculations. Finally, $i$ consists of both the embeddings of $\bar{C}_H(Q)$ and $\bar{C}_G(Q)$ inside $N_{\Delta}$. 

\begin{verbatim}
function StableEqSetup(G,H,NG,Q);
/* Here we define all groups and subgroups that are involved 
in the construction of the stable equivalence.
We make sure that each subgroup is generated by two elements. */ 
NH:=Normaliser(H,Q);
NH:=GenTwoEl(NH);
CG:=Centraliser(G,Q);
CG:=GenTwoEl(CG);
CH:=Centraliser(H,Q);
CH:=GenTwoEl(CH);
BNH:=Complements(NH,Q)[1];
BNH:=GenTwoEl(BNH);
BNG:=Complements(NG,Q)[1];
BNG:=GenTwoEl(BNG);
/* As requested by the algorithm, BNH must be contained in BNG */ 
repeat 
      g:=Random(NG); BNG:=Conjugate(BNG,g); 
until BNH subset BNG;
BCH:=CH meet BNH; 
BCH:=GenTwoEl(BCH);
BCG:=CG meet BNG;
BCG:=GenTwoEl(BCG);
D,i,p:=DirectProduct(NH,NG);
ND:=sub<D|i[1](BCH.1),i[1](BCH.2),i[2](BCG.1),i[2](BCG.2),
i[1](BNH.1)*i[2](BNH.1),i[1](BNH.2)*i[2](BNH.2)>;
DP:=sub<ND|i[1](BCH.1),i[1](BCH.2),i[2](BCG.1),i[2](BCG.2)>;
IBNH:=sub<ND|i[1](BNH.1)*i[2](BNH.1),i[1](BNH.2)*i[2](BNH.2)>;
IBCH:=sub<ND|i[1](BCH.1),i[1](BCH.2)>;
IBCG:=sub<ND|i[2](BCG.1),i[2](BCG.2)>;
/* We can now define the k[BCH]-k[BCG] bimodule k[BCG] */ 
LG:=[g : g in IBCG];
n:=#IBCG;
k:=GF(#Q); 
Zg1:=ZeroMatrix(k,n,n);
for i in LG do 
       Zg1[Position(LG,i),Position(LG,i*IBCG.1)]:=1; 
end for;
Zg2:=ZeroMatrix(k,n,n);
for i in LG do 
       Zg2[Position(LG,i),Position(LG,i*IBCG.2)]:=1; 
end for;
Zh1:=ZeroMatrix(k,n,n);
for j in LG do 
       Zh1[Position(LG,j),Position(LG,i[2](BCH.1^(-1))*j)]:=1; 
end for;
Zh2:=ZeroMatrix(k,n,n);
for j in LG do 
      Zh2[Position(LG,j),Position(LG,i[2](BCH.2^(-1))*j)]:=1; 
end for;
/* Here we define the action of \bar{N_H(Q)}, so kC_G(Q) is a module 
for the whole N_{\Delta} */ 
Zn1:=ZeroMatrix(k,n,n);
for i in LG do 
       Zn1[Position(LG,i),Position(LG,(IBNH.1)^(-1)*i*IBNH.1)]:=1; 
end for;
Zn2:=ZeroMatrix(k,n,n);
for i in LG do 
       Zn2[Position(LG,i),Position(LG,(IBNH.2)^(-1)*i*IBNH.2)]:=1; 
end for;
/* N_{\Delta} has 6 generators: 2 for C_H(Q), 2 for C_G(Q), 
and 2 for the diagonal \bar{N_H(Q)} */ 
V:=GModule(ND,[Zh1,Zh2,Zg1,Zg2,Zn1,Zn2]); 
return V,BCG,IBCG,NH,BNH,IBNH,i,IBCH;
end function;
\end{verbatim}

The following algorithm is the main one. This will be used to compute $T_Q \otimes_{k\bar{C}_G(Q)} S'$; we will typically run it when $S'$ is indecomposable as a $k N_G(Q)$-module.

\begin{verbatim}
function StableEquivalence(Tq,V,H,Q,BCG,IBCG,NH,BNH,IBNH,i);
ND:=Group(Tq);
NG:=Group(V);
Gamma,ii,pp:=DirectProduct(ND,NG);
g:=NH.1;
for x in BNH do
    if x*g^(-1) in Q then y1:=x;
    end if;
end for;  
g:=NH.2;
for x in BNH do
    if x*g^(-1) in Q then y2:=x;
    end if;
end for;  
s:=hom< NH -> IBNH|i[1](y1)*i[2](y1),i[1](y2)*i[2](y2)>;
/* s is the "quotient" map of NH onto the diagonal copy of BNH 
inside ND=N_{Delta}. */ 
/* x1, x2 generate N_H(Q) inside Gamma, and we recall that 
ii is the embedding of N_{Delta} and N_G(Q) inside Gamma. */
x1:=ii[1](s(NH.1))*ii[2](NH.1);
x2:=ii[1](s(NH.2))*ii[2](NH.2);
/* Finally, the copy of N_H(Q) which is diagonally 
embedded inside Gamma: */
NNH:=sub<Gamma|x1,x2>;
k:=Field(Tq);
/* We have V, which is a N_G(Q)-mod, and now we provide it with the 
(trivial) action of the other factor of Gamma, i.e. N_{Delta}. */
d:=Dimension(V);
IdV:=IdentityMatrix(k,d);
a:=ActionGenerators(V);
NewV:=GModule(Gamma,[IdV,IdV,IdV,IdV,IdV,IdV,a[1],a[2]]); delete a;
/* We have T_q now, which is a N_{Delta}-mod, and we give it the 
(trivial) action of the other factor of Gamma, i.e. N_G(Q). */
d:=Dimension(Tq);
IdTq:=IdentityMatrix(k,d);
a:=ActionGenerators(Tq);
NewTq:=GModule(Gamma,[a[1],a[2],a[3],a[4],a[5],a[6],IdTq,IdTq]); 
delete a;
/* Generators of the centraliser. 
We need them for the relations that we quotient by. */
a1:=ii[1](IBCG.1)^(-1);
b1:=ii[2](BCG.1);
a2:=ii[1](IBCG.2)^(-1);
b2:=ii[2](BCG.2);
Ten:=TensorProduct(NewTq,NewV); 
ListT1:=[]; ListT2:=[];
ListV1:=[]; ListV2:=[];
ResTq:=Restriction(NewTq,ii[1](IBCG));
ResV:=Restriction(NewV,ii[2](BCG));
IT:=IndecomposableSummands(ResTq);
print "\nRestricted to the Q-complement of CG(Q), the module Tq 
decomposes into", #IT, "summands of dimension:"; 
l:=[];
for x in IT do Append(~l,Dimension(x));
end for;
l;
IV:=IndecomposableSummands(ResV);
print "\nRestricted to the Q-complement of CG(Q), the module S 
decomposes into", #IV, "summands of dimension:"; 
l:=[];
for x in IV do Append(~l,Dimension(x));
end for;
l;
"\nNow Tq and L have been decomposed as much as possible, 
namely the action is restricted to the Q-complement of CG(Q).";
NewBasisTq:=[];
/* We create vectors of the shape tg*l-t*gl, where * is 
tensor product. We make two lists, i.e. vectors t*g's and g*l's. */
for C in IT do
    basC:=Basis(C);
    NewBasisTq:=NewBasisTq cat [NewTq!(ResTq!v) : v in basC];
    LC1:=[NewTq!(ResTq!(v*a1)) : v in basC];
    LC2:=[NewTq!(ResTq!(v*a2)) : v in basC];
    ListT1:=ListT1 cat LC1;
    ListT2:=ListT2 cat LC2;
end for;
"Done with Tq.";
NewBasisV:=[];
for D in IV do
    basD:=Basis(D);
    NewBasisV:=NewBasisV cat [NewV!(ResV!v) : v in basD];
    LD1:=[NewV!(ResV!(v*b1)) : v in basD];
    LD2:=[NewV!(ResV!(v*b2)) : v in basD];
    ListV1:=ListV1 cat LD1;
    ListV2:=ListV2 cat LD2;
end for;
"Done with S, we have our vectors in Tq and S, now we tensor them.";  
/* ListT1, ListT2 are coerced vectors in NewTq; ListV1, ListV2 are 
vectors of NewV. Now we tensor them, so we get our set of desired 
vectors in NewTq x NewV, namely Ten */
ListTen1:=[];
ListTen2:=[];
m:=0;
for i in [1..#ListT1] do 
    for j in [1..#NewBasisV] do 
        Append(~ListTen1,Ten!Vector((TensorProduct(ListT1[i],NewBasisV[j])-
        TensorProduct(NewBasisTq[i],ListV1[j]))));       
        m:=m+1; 
        if (m mod 1000) eq 0 then  
           "We have tensored", m, "vectors out of", 2*#ListT1*#NewBasisV;  
        end if;
    end for;
end for;
for i in [1..#ListT2] do 
    for j in [1..#NewBasisV] do 
        Append(~ListTen2,Ten!Vector((TensorProduct(ListT2[i],NewBasisV[j])-
        TensorProduct(NewBasisTq[i],ListV2[j]))));
        m:=m+1; 
        if (m mod 1000) eq 0 then  
           "We have tensored", m, "vectors out of", 2*#ListT1*#NewBasisV;  
        end if;
    end for;
end for;
"\nNow we generate our submodule, quotient, clean off projectives, 
and return the final kN(D)-module.";
ListFinal:=ListTen1 cat ListTen2;
"Now we restrict the tensor product to N_H(Q), 
its dimension is", Dimension(Ten); 
Ten:=Restriction(Ten,NNH);
Rel:=sub<Ten|ListFinal>;
Xs:=Ten/Rel;
r:=Representation(Xs);
_,f:=IsIsomorphic(NNH,Group(Xs));
U:=GModule(NH,[r(f(NNH.1)),r(f(NNH.2))]); 
p:=#Q;
ProjU:=[ProjectiveCover(x) : x in IrreducibleModules(Group(U),GF(p))];
n:=Dimension(U) div #Group(U);
U:=RemFree(U,2*n);
U:=RemoveAllProj(U,ProjU);
IV:=Induction(U,H);
return IV;
end function;
\end{verbatim}

The final algorithm aims to iterate the previous algorithm \texttt{StableEquivalence} over each indecomposable summands of the $k N_G(Q)$-module $S$. We will use \texttt{SplitL} first and we will process the non-projective part of $S$ first, as most of the times the projective summands have been already processed in a previous case and there is no need to redo the calculation. The list of projective indecomposable $kN_G(Q)$-modules is ProjNG. Whether we want to process the projective summands of $S$ as well or not, it is decided by the input ``bool''. 
\begin{verbatim}
function FinalStabEq(Tq,S,H,Q,BCG,IBCG,NH,BNH,IBNH,i,ProjNG,bool);
/* Here bool decides if we have to compute the tensor of Tq with the 
projective summands of S as well. Sometimes, we already know those, 
as it was done before, and we do not have to do the same computation 
again, in this case we set bool=false. */
NG:=Group(S);
T,LN:=SplitL(S,ProjNG);
/* Let us count how many summands S splits into. 
We will print this result on screen. */ 
c:=0; 
NonZero:=[[Dimension(T),1]];
/* Let us remember that LN is the list of multiplicities of 
indecomposable projective inside L. 
The index h runs across the total number of projectives. */ 
for h in [1..#ProjNG] do 
    if not (LN[h] eq 0) then 
       c:=c+1; Append(~NonZero,[Dimension(ProjNG[h]),LN[h]]);
    end if;
end for;
print "\nThe kN(Q)-module decomposes into summands 
of dimension (with multiplicities):"; 
for x in NonZero do x;
end for;
/* First, we find the desired tensor of Tq with the non-projective 
part of L. We will add the "projective" part later. */ 
print "\nWe work on the tensor no", 1, "out of", #NonZero;
U:=StableEquivalence(Tq,T,H,Q,BCG,IBCG,NH,BNH,IBNH,i);
if not bool then 
   return U;
end if;
/* Whenever bool=true, we go on and now we sum the contribution 
coming from the projective summands of L. */
num:=2;
for j in [1..#ProjNG] do 
    if not (LN[j] eq 0) then
       print "\nWe work on the tensor no", num, "out of", #NonZero;   
       StEq:=StableEquivalence(Tq,ProjNG[j],H,Q,BCG,IBCG,NH,BNH,IBNH,i);
       for k in [1..LN[j]] do 
           U:=DirectSum(U,StEq);
       end for;
    num:=num+1;
    end if;
end for;
return U;
end function;

\end{verbatim}

\section{Appendix B: the algorithm \texttt{PerverseEq}} 

We recall that \texttt{PerverseEq} is the algorithm taking a $B_0(H)$-module $T_i$ as input, and constructing a complex $X_i$ which is meant to be the image of $T_i$ under a perverse derived equivalence, provided that the conditions of Proposition \ref{forestanera} are fulfilled.

\vspace{3mm}

For a $kG$-module $U$ and a list $X$ of simple $kG$-modules, the following algorithm returns the maximal \emph{semisimple} submodule $V \subseteq U$ with composition factors in the list $X$; notice that the set of composition factors of the zero-module is the empty subset of $X$, and therefore such a submodule always exists. 
\begin{verbatim}
function SemisimpleXRad(Q,X);
K:=[];
for M in X do
    hom:=AHom(M,Q);
    if Dimension(hom) gt 0 then
       B:=&+[Image(hom.j) : j in [1..Dimension(hom)]];
       K:=Append(K,B);
    end if;
end for;
if #K eq 0 then
   return sub<Q|0>;
else
   T:=&+K; return T;
end if;
end function;
\end{verbatim}
The following is a straight application of the previous one. Given a list of simple $kG$-modules $X$, a $kG$-module $U$ and a submodule $V$, the function returns $W$ such that $V \subseteq W \subseteq U$ and $W/V$ is the $X$-radical of $U/V$. This is equivalent to saying that $W$ is the maximal submodule such that $V \subseteq W \subseteq U$ and there is a filtration from $V$ to $W$ whose quotients are in $X$.
\begin{verbatim}
function PreImageXRadical(P,M,X);
Q,q:=P/M; N:=M;
_,R:=HasPreimage(SemisimpleXRad(Q,X),q);
/* If N equal R, we do not enter the loop. Indeed, 
it means that there is nothing acceptable between M and P, 
so it returns M itself as M/M is considered to be semisimple */
while Dimension(N) lt Dimension(R) do
      N:=R; Q,q:=P/N; 
      _,R:=HasPreimage(SemisimpleXRad(Q,X),q);
end while;
return R;
end function;
\end{verbatim}
\vspace{2mm}
Now let $n \in \ZZ_{\geq 0}$ and $p: \mathcal{S}_{B_0(H)} \to \ZZ_{\geq 0}$. Here we get the set $J_n$ defined in \ref{algopi}.
\vspace{2mm}
\begin{verbatim}
function J(X,n,p);
I:={};
for M in X do
    if p(M) le n then
       I:=Include(I,M);
    end if;
end for;
return I;
end function;
\end{verbatim}
\vspace{2mm}
The following returns the injective hull of a $kG$-module $M$ equipped with an injective map.
\vspace{2mm}
\begin{verbatim}
function InjHull(M);
IM:=Dual(ProjectiveCover(Dual(M))); h:=AHom(M,IM);
repeat f:=Random(h); 
until  IsInjective(f);
return IM,f;
end function;
\end{verbatim}
\vspace{2mm}

We are now finally able to build the algorithm $\texttt{PerverseEq()}$. We recall that this algorithm will allow us to deduce if Rouquier's stable equivalence can be lifted to a perverse derived equivalence with a given bijections between $B_0(H)$ and $B_0(H)$ and a given perversity function. Given $T_i \in \mathcal{S}_{B_0(H)}$, we remark that $\texttt{PerverseEq()}$ is physically building the complex $X_i$ from degree $-n$ to degree $-2$, as well as a submodule of the term of degree $-1$, which is meant to be the kernel of the last map of the complex; the crucial term in degree $-1$ can be built manually using the MAGMA commands $\texttt{Ext}()$ and $\texttt{Extension}()$ as we explained in Section \ref{saa}, or using the code \texttt{FindP1}. We also recall that if the algorithm is successful for every simple $B_0(H)$-module $T_i$, then the complex $X_i$ is the image of $S_i \in \mathcal{S}_{B_0(H)}$ under the perverse equivalence between $\mathcal{D}(B_0(G))$ and $\mathcal{D}(B_0(H))$ that we have just found. 

Hence, $T$ is a simple $B_0(H)$-module, $X$ denotes $\mathcal{S}_{B_0(H)}$ and $p: X = \mathcal{S}_{B_0(H)}\to \ZZ_{\geq 0}$ is a (perversity) function. Sequences of the kernels, images and cohomologies are also returned. 

\begin{verbatim}
function PerverseEq(T,p,X);
if p(T) eq 0 then
   return "The complex is trivial, T->0";
end if;
n:=p(T); P:=[]; K:=[]; I:=[];
P[1],i:=InjHull(T); T:=Image(i);
K[1]:=PreImageXRadical(P[1],T,L(X,n-1,p));
for r in [2..n] do
    B:=P[r-1]/K[r-1];
    P[r],i:=InjHull(B);
    I[r-1]:=Image(i);
    Q,q:=P[r]/I[r-1];
    K[r]:=PreImageXRadical(P[r],I[r-1],L(X,n-r,p));   
end for;
H:=[K[1]];
for r in [2..n] do
    Append(~H,K[r]/I[r-1]);
end for;
P:=Prune(P);
return P,K,H,I;
end function;
\end{verbatim}

This algorithm builds the complex $X_T$ up to degree $-2$, as well as the submodule $M_1$ of the module in degree $-1$ (which is still unknown) $P_1$; it will be the kernel of the map $P_1 \to P_0$. We need to build $P_1$ and $P_0$. As we explained in Section \ref{atee}, $P_1$ must be an extension of $M_1$ by $C_S$, where $S$ corresponds to $T$ via the bijection between $\mathcal{S}_{B_0(H)}$ and $\mathcal{S}_{B_0(G)}$, and $C_S$ is the Green correspondent of $S$ down to $H$. Such extension must be the direct sum of stacked relatively $Q$-projective modules for some $Q<P$. MAGMA provides the function \texttt{Ext} to compute the space $\text{Ext}^1(C_S,P_1)$. Each vector $v \in \text{Ext}^1(C_S,P_1)$ provides an extension
 $E_v$ of $P_1$ by $C_S$, namely $P_1 \subset E_v$ and $E_v/P_1 \cong C_S$. In MAGMA, we can access and use the vectors of $\text{Ext}^1(C_S,P_1)$ by defining the vector space and the map \texttt{V,r:=Ext(A,B)} and for a given vector \texttt{v:=Random(V)}, we create the extension $E_v$ by setting \texttt{E:=Extension(A,B,v,r)}. The algorithm \texttt{FindP1} works via the following stages:

\begin{itemize}
\item Each vector $v$ of the vector space Ext provides an extension of $P_1$ by $C_i$, and some of them will be isomorphic; for example, each of $2v, \dots, (\ell-1)v$ will provide the same extension $E_{v}$ that $v$ does. Therefore we reduce the set of vectors \texttt{V} to \texttt{Set} and we create the list of extensions $E_v$ for each $v$ in the new set. We reduce this list by using \texttt{IsomorphismClasses}, which deletes redundant isomorphic copies out of the list. 
\item We run over each $E$ of this list and we decompose it. We are only interested in extensions who decompose into projective or stacked relatively projective modules with respect to some $Q \simeq C_{\ell}$. Projective modules have dimension divisible by $|D|=\ell^2$, whereas in the second case they have dimension which is multiple of the dimension of the indecomposable relatively $Q$-projective modules of trivial source. This number is given as an input for the function (denoted by $n$; for $\Omega_{8}^{+}(2)$, we have $n=15$). We delete every $E$ whenever one of its indecomposable summand does not fulfil those requirements. This is generally enough to considerably reduce the set of potential extensions to check.  
\end{itemize}
In most cases, the vector space $V$ has dimension $1$ or $2$ (namely: $\ell$ or $\ell^2$ elements) and the following function is not actually necessary. However, in other cases (like some sporadic groups) $V$ can be larger, and therefore such function can be useful to detect if some valid extension exists at all or not. 
\begin{verbatim}
function FindP1(V,r,K,GC,n,l);
Set:=[];
Rubbish:=[V!0];
/* here we get rid of useless vector, i.e. multiples */ 
for v in V do 
    if not (v in Rubbish) then Append(~Set,v);
    end if;
    for i in [1..l-1] do 
        Append(~Rubbish,i*v);
    end for;
end for;
"We got rid of repetitions in V. Now we compute extensions.";
EE:=[Extension(GC,K,v,r) : v in Set]; 
"Extensions computed. Now we reduce it, removing isomorphic copies.";
EE:=IsomorphismClasses(EE);
"Done. Now we decompose them and store whatever has a suitable decomposition."; 
Decompos:=[];
for E in EE do 
    I:=IndecomposableSummands(E); t:=true;
    for M in I do 
        if not (IsProjective(M) or Dimension(M) mod n eq 0) then t:=false;
        end if;
    end for;
    if t eq true then Append(~Decompos,I); delete I; 
    end if;
end for;
return Decompos;
end function;
\end{verbatim}

\section{Acknowledgements} 
I would like to thank David Craven, who introduced me into the field of Modular Representation Theory and supervised my PhD work at the University of Birmingham. I am also grateful for the financial support that I have received from the Engineering and Physical Sciences Research Council. 

This article has been written during my stay at the University of Kaiserslautern; I want to thank Gunter Malle for giving me the opportunity to be here, as well as the research centre  SFB TRR 195 and the University of Kaiserslautern for the financial support. 

Finally, I thank the referee for the valuable corrections and suggestions that have made the article more clear.

\bibliographystyle{amsplain}

\end{document}